\documentclass[11pt,a4paper]{article}
\usepackage[left=2.0cm, top=2.4cm,bottom=2.4cm,right=2.0cm]{geometry}
\usepackage{mathtools,amssymb,amsthm,mathrsfs,calc,graphicx,xcolor,cleveref,dsfont,tikz,pgfplots,bm,url}
\usepgfplotslibrary{fillbetween}
\usetikzlibrary{patterns}
\usepackage[british]{babel}
\usepackage{amsfonts}              
\usepackage[T1]{fontenc}
\usepackage{enumitem}
\numberwithin{equation}{section}
\numberwithin{figure}{section}

\usepackage{float}
\restylefloat{table}

\theoremstyle{plain}
\newtheorem{theorem}{Theorem}[section]
\newtheorem{lemma}[theorem]{Lemma}
\newtheorem{corollary}[theorem]{Corollary}

\theoremstyle{definition}

\theoremstyle{remark}
\newtheorem{remark}[theorem]{Remark}

\def\BB{\mathbb{B}}

\def\EE{\mathbb{E}}

\def\MM{\mathbb{M}}
\def\NN{\mathbb{N}}

\def\PP{\mathbb{P}}

\def\RR{\mathbb{R}}

\def\XX{\mathbb{X}}

\def\ZZ{\mathbb{Z}}

\def\sfN{{\sf N}}

\def\cC{\mathcal{C}}
\def\cD{\mathcal{D}}

\def\cF{\mathcal{F}}

\def\cL{\mathcal{L}}
\def\cM{\mathcal{M}}
\def\cN{\mathcal{N}}
\def\cO{\mathcal{O}}

\def\cX{\mathcal{X}}

\newcommand{\sk}[1]{\mathscr{#1}}

\newcommand{\dd}{{\rm d}}
\newcommand{\conv}{\mathop{\mathrm{conv}}\nolimits}
\newcommand{\aff}{\mathop{\mathrm{aff}}\nolimits}
\newcommand{\pow}{\mathop{\mathrm{pow}}\nolimits}

\newcommand{\inter}{\operatorname{int}}
\newcommand{\ext}{\operatorname{ext}}
\newcommand{\ver}{\operatorname{vert}}
\newcommand{\Vol}{\operatorname{Vol}}
\newcommand{\skel}{\operatorname{skel}}

\newcommand{\dint}{\textup{d}}
\newcommand{\apex}{\mathop{\mathrm{apex}}\nolimits}

\makeatletter
\let\@fnsymbol\@alph
\makeatother

\begin{document}
	
	\title{\bfseries The $\beta$-Delaunay tessellation II:\\ The Gaussian limit tessellation}
	
	\author{Anna Gusakova\footnotemark[1],\; Zakhar Kabluchko\footnotemark[2],\; and Christoph Th\"ale\footnotemark[3]}
	
	\date{}
	\renewcommand{\thefootnote}{\fnsymbol{footnote}}
	\footnotetext[1]{Ruhr University Bochum, Germany. Email: anna.gusakova@rub.de}
	
	\footnotetext[2]{M\"unster University, Germany. Email: zakhar.kabluchko@uni-muenster.de}
	
	\footnotetext[3]{Ruhr University Bochum, Germany. Email: christoph.thaele@rub.de}
	
	\maketitle

\begin{abstract}
\noindent We study the weak convergence of $\beta$- and $\beta'$-Delaunay tessellations in $\mathbb{R}^{d-1}$ that were introduced in part I of this paper, as $\beta\to\infty$. The limiting stationary simplicial random tessellation, which is called the Gaussian-Delaunay tessellation, is characterized in terms of a space-time paraboloid hull process in $\mathbb{R}^{d-1}\times\mathbb{R}$. The latter object has previously appeared in the analysis of the number of shocks in the solution of the inviscid Burgers' equation and the description of the local asymptotic geometry of Gaussian random polytopes. In this paper it is used to define a new stationary random simplicial tessellation in $\mathbb{R}^{d-1}$. As for the $\beta$- and $\beta'$-Delaunay tessellation, the distribution of volume-power weighted typical cells in the Gaussian-Delaunay tessellation is explicitly identified, establishing thereby a new bridge to Gaussian random simplices. Also major geometric characteristics of these cells such as volume moments, expected angle sums and also the cell intensities of the Gaussian-Delaunay tessellation are investigated.\\
		
	\end{abstract}

\small
\tableofcontents
\normalsize
\section{Introduction}

In part I of this paper \cite{Part1} we introduced two new classes of stationary random simplicial tessellations in $\RR^{d-1}$, the $\beta$-Delaunay tessellation (with $\beta>-1$) and the $\beta'$-Delaunay tessellation (with $\beta>(d+1)/2$). Their construction, which we recall in detail in Section \ref{subsec:betaDelaunay} below, is based on a paraboloid growth and hull processes, which were originally introduced in \cite{CSY13,SY08} in order to describe the asymptotic geometry of random convex hulls in the unit ball. The definition of the paraboloid growth and hull processes in turn is based on a space-time Poisson point process on the product space $\RR^{d-1}\times\RR_+$ (for the $\beta$-Delaunay tessellation) or $\RR^{d-1}\times\RR_-^*:=\RR^{d-1}\times(-\infty,0)$ (for the $\beta'$-Delaunay tessellation) whose intensity measure is translation invariant in the spatial (i.e.\ $\RR^{d-1}$) component and has a power-law density function in the time or height (i.e.\ $\RR_+$ or $\RR_-^*$) coordinate. We used these processes to generalize the well-known construction of the Poisson-Delaunay tessellation, which is one of the most classical tessellation models considered in stochastic geometry and related fields. Moreover, in \cite{Part1} we gave an explicit description of weighted typical cells of both the $\beta$- and the $\beta'$-Delaunay tessellations from which we computed several key characteristic quantities such as volume moments, expected angle sums and also cell intensities. In particular, this demonstrates that both the $\beta$- and $\beta'$-Delaunay tessellations are analytically tractable, which makes them attractive alternatives to the classical Poisson-Delaunay tessellation. Let us point out that the latter are widely used in wireless network modelling \cite{BarthelemyNetwork,GlisicWireless,JahnelKoenig} or material sciences \cite{OstojaMaterial}.

In the present paper we study the limiting behaviour of the sequence of $\beta$- and $\beta'$-Delaunay tessellations in $\RR^{d-1}$, as $\beta\to\infty$. Under a suitable rescaling, one can verify (see Lemma \ref{lm:processconverge} below) that the Poisson point processes underlying the constructions of the paraboloid growth and hull processes (and hence the $\beta$- and $\beta'$-Delaunay tessellation) converge in distribution to a Poisson point process on $\RR^{d-1}\times\RR$ whose intensity measure is again translation invariant in the spatial component and has density $e^{h/2}$ in the time or height coordinate. Such Poisson point processes together with their parabolic hulls have previously appeared in the literature, namely in the expectation and variance asymptotics for the number of shocks in the solution of the inviscid Burgers equation \cite{BaryshnikovBurgers}, as well as in the description of geometric properties of convex hulls of Gaussian random points~\cite{CYGaussian}. A closely related construction appeared as an example of a max-stable process in~\cite{eddy_gale,smith}. We use this point process to construct the related paraboloid growth and hull process and to define another stationary random simplicial tessellation in $\RR^{d-1}$, which we call the Gaussian-Delaunay tessellation. Given these constructions, it seems natural to expect a similar (functional) convergence on the level of the random tessellations themselves, i.e., that, as $\beta\to\infty$, both the $\beta$- and $\beta'$-Delaunay tessellation weakly converge to the Gaussian-Delaunay tessellation. Of course, in order to make this rigorous, one needs to make precise what weak convergence of random tessellations formally means. However, this is by far not straightforward, since there is no natural topological structure on the space of tessellations and it is not clear whether the space of tessellations of $\RR^{d-1}$ is even Polish. This turns problems related to the convergence of random tessellations into a highly non-trivial task, which has only very rarely been tackled in the existing literature (see \cite{NagelWeissSTIT1,NagelWeissSTIT2} for the only example we are aware of). One way to work around this problem is by identifying a tessellation with its skeleton, by which we mean the union of the boundaries of all of its cells. From here on one can work with the well-established notion of weak convergence of random closed sets with respect to the usual Fell topology, which is conveniently handled by means of the so-called capacity functional, see \cite[Chapter 16]{Kallenberg} or \cite[Chapter 2]{SW}. Eventually, we deal with the capacity functional of the $\beta$- and $\beta'$-Delaunay tessellation by resorting to the parabolic growth and hull processes underlying their constructions and by proving a kind of localization property for these processes, which is very much in the spirit of the well established geometric limit theory of stabilization for which we refer to the survey articles \cite{SchreiberSurvey,YukichSurvey}. This constitutes the most technical part of this paper, in which we show that, after suitable rescaling, the capacity functional of the $\beta$- and $\beta'$-Delaunay tessellation converges to that of the Gaussian-Delaunay tessellation, proving thereby the weak convergence of the random tessellations mentioned earlier. In addition, our proof shows convergence in a topology which is much stronger than the usual Fell topology on the space of closed subsets of $\RR^{d-1}$. This is of interest, because quantities like the number of cells or the total surface content in a bounded window become continuous functionals in this stronger topology, but are not continuous with respect to the Fell topology. This new perspective on the weak convergence of random tessellations (or more general random closed sets) is of independent interest and might be fruitful also in the study of functional limit theorems of other stochastic geometry models beyond the ones we investigate in this paper.

It is another goal of this paper to discuss and to determine principal geometric properties and characteristics of Gaussian-Delaunay tessellations. For that reason we follow the lines of part I \cite{Part1}, introduce the class of volume-power weighted typical cells and compute explicitly their volume moments, expected angle sums and cell intensities. In particular, this class of cells includes the typical cell (in the usual sense of Palm calculus) and, up to translation, the zero-cell (the almost surely uniquely determined cell containing the origin) of the Gaussian-Delaunay tessellation. The results we obtain in this context will be a consequence of a probabilistic description of volume-power weighted typical cells in terms of Gaussian random simplices. As for the $\beta$- and the $\beta'$-Delaunay tessellation this demonstrates that also the new Gaussian-Delaunay tessellation is an analytically tractable model of a simplicial random tessellation. In particular, the fact that the typical cell has the same distribution is a volume-weighted Gaussian random simplex is the reason why we call the limiting object the Gaussian-Delaunay tessellation.

The remaining parts of this paper are organized as follows. In Section \ref{sec:Prelim} we gather some necessary background material, especially about point processes and tessellations (Section \ref{subsec:PPTess}) and the construction of Laguerre diagrams and tessellations (Section \ref{sec:Laguerre_tess}). In Section \ref{sec:Constructions} we explain the paraboloid growth and hull processes that lead to the definition of $\beta$-, $\beta'$- and Gaussian-Delaunay tessellations. Our main result about the convergence of $\beta$- and $\beta'$-Delaunay tessellations towards the Gaussian-Delaunay tessellation is presented and proved in Section \ref{sec:Convergence}. In the final Section \ref{sec:Cells} a probabilistic representation of volume-power weighted typical cells is found and some key geometric characteristics of such cells are computed.

\section{Preliminaries}\label{sec:Prelim}

\subsection{Notation}

A centred unit Euclidean ball in $\RR^d$ is denoted by $\BB^d$ and its volume is given by
$$
\kappa_d:=\frac{\pi^{d/2}}{\Gamma(1+{d\over 2})}.
$$
By $B_{r}(w)$ we denote the closed $(d-1)$-dimensional ball of radius $r>0$ centred at $w\in\RR^{d-1}$. We will also use a simplified notation for the ball $B_r:=B_r(0)$ centred at $w=0$. 

In what follows we will often use two equivalent systems of coordinates, namely, we represent points $x\in\RR^d$ as $x:=(v,h)\in \RR^d$ with $v\in\RR^{d-1}$ (usually called the spatial coordinate) and $h\in\RR$ (referred to as the height or time coordinate).   Finally, the Euclidean norm is indicated by $\|\,\cdot\,\|$.

Given a set $C$ in a Euclidean space we denote by $\conv(C)$ its convex hull, by ${\rm int}(C)$ its interior and by $\partial C$ its boundary. We denote by $\RR_{+}$ the set of all non-negative real numbers and by $\RR_{-}^*:=(-\infty,0)$ the set of all negative real numbers.

For two functions $f,g:\RR^k\to\RR$, $k\in\NN$, we will frequently use the notation $f(x_1,\ldots,x_k)\ll_{c_1,c_2,\ldots} g(x_1,\ldots,x_k)$, which means that there exists a positive constant $c$, which depends on parameters $c_1,c_2,\ldots$ but is independent of $(x_1,\ldots,x_k)\in\RR^k$, such that $f(x_1,\ldots,x_k)\leq c g(x_1,\ldots,x_k)$ for all $(x_1,\ldots,x_k)\in\RR^k$.

\subsection{General point processes and tessellations}\label{subsec:PPTess}

Let $(\XX,\cX)$ be a measurable space  supplied with a $\sigma$-finite measure $\mu$. By $\sfN(\XX)$ we denote the space of $\sigma$-finite counting measures on $\XX$. The $\sigma$-field $\cN(\XX)$ is defined as the smallest $\sigma$-field on $\sfN(\XX)$ such that the evaluation mappings $\xi\mapsto\xi(B)$, $B\in\cX$, $\xi\in\sfN(\XX)$, are measurable. A \textbf{point process} on $\XX$ is a measurable mapping with values in $\sfN(\XX)$ defined over some fixed probability space $(\Omega,\cF,\PP)$. By a \textbf{Poisson point process} $\eta$ on $\XX$ with intensity measure $\mu$ we understand a point process with the following two properties:
\begin{itemize}
    \item[(i)] for any $B\in\cX$ the random variable $\eta(B)$ is Poisson distributed with mean $\mu(B)$;
    \item[(ii)] for any $n\in\NN$ and pairwise disjoint sets $B_1,\ldots,B_n\in\cX$ the random variables $\eta(B_1),\ldots,\eta(B_n)$ are independent.
\end{itemize}
We refer to \cite[Chapter 3]{SW} for further background material on Poisson point processes.

Next, we introduce the notion of a general random tessellation in $\RR^{d-1}$ and give a brief overview of some basic properties, which are used throughout the paper. For a more detailed discussion we refer to \cite[Chapter 10]{SW}. A \textbf{tessellation} $\cM$ in $\RR^{d-1}$ is a locally finite countable system of compact, convex subsets of $\RR^{d-1}$ having interior points, covering the space and having disjoint interiors. The elements of a tessellation $\cM$ are called {\bf cells} and every cell is a convex polytope. Given a polytope $P\subset \RR^{d-1}$ denote by $\cF_k(P)$ the set of its $k$-dimensional faces and let $\cF(P):=\bigcup_{k=0}^{d-1}\cF_{k}(P)$ be the set of all faces of $P$. A tessellation $\cM$ is called {\bf face-to-face} if for all $P_1,P_2\in \cM$ we have
$$
P_1\cap P_2\in(\cF(P_1)\cap\cF(P_2))\cup\{\varnothing\}.
$$
A face-to-face tessellation in $\RR^{d-1}$ is called {\bf normal} if each $k$-dimensional face of the tessellation is contained in precisely $d-k$ cells, $k\in\{0,1,\ldots,d-2\}$. Further, we denote by $\MM$ the set of all face-to-face tessellations in $\RR^{d-1}$. By a {\bf random tessellation} in $\RR^{d-1}$ we understand a particle process $X$ in $\RR^{d-1}$ (in the usual sense of stochastic geometry, see \cite{SW}) satisfying $X\in\MM$ almost surely.

\subsection{Laguerre tessellations}\label{sec:Laguerre_tess}

One of the most well studied type of tessellations is the classical Voronoi tessellation, see, for example, \cite{OkabeEtAl,SKM,SW}. A Laguerre tessellation can be considered as a generalized (or weighted) version of a Voronoi tessellation and was intensively studied in \cite{LZ08, Ldoc, Sch93}. In this subsection we only briefly recall some facts about Laguerre tessellations. For more details we refer the reader to part I of this paper \cite[Section 3.2 - 3.4]{Part1}.

The construction of a Laguerre diagram is based on the notion of a power function. For $v,w \in \RR^{d-1}$ and $W\in\RR$ we define the power of $w$ with respect to the pair $(v,W)$ as
\[
\pow (w,(v,W)):=\|w-v\|^2+W,
\]
where $W$ is referred to as the weight of the point $v$. Let $X\subset\RR^{d-1}\times\RR$ be a countable set of weighted points in $\RR^{d-1}$ such that $\min_{(v,W)\in X}\pow(w,(v,W))$ exists for each $w\in\RR^{d-1}$. Then the \textbf{Laguerre cell} of $(v,W)\in X$ is defined as
\[
C((v,W),X):=\{w\in\RR^{d-1}\colon \pow(w,(v,W))\leq \pow(w,(v',W'))\text{ for all }(v',W')\in X\}.
\]
We emphasize that it is not necessarily the case that a Laguerre cell is non-empty or that it contains interior points. The collection of all Laguerre cells of $X$ with non-vanishing topological interior is called the \textbf{Laguerre diagram} of $X$ and we write
\[
\cL(X):=\{C((v,W),X)\colon (v,W)\in X,\; {\rm int}(C((v,W),X))\neq\varnothing\}.
\]

It should be mentioned that a Laguerre diagram is not necessarily a tessellation, since the latter property strongly depends on the geometric properties of the point set $X$. In part I of this paper we proved the following lemma, which gives a sufficient condition for a point process $\xi$ in $\RR^{d-1}\times E$, $E\subset \RR$ a Borel set, which ensures that $\cL(\xi)$ is in fact almost surely a random tessellation.

\begin{lemma}[Lemma 3.1 and Lemma 3.2 in \cite{Part1}]\label{lm:LTessel}
Let $\xi$ be a point process in $\RR^{d-1}\times E$, satisfying the following conditions.

\begin{itemize}
\item[(P1)] Almost surely we have
$$
\conv(v\colon (v,h)\in\xi)=\RR^{d-1}.
$$
\item[(P2)] For every $w\in\RR^{d-1}$ and every $t\in E$ there are almost surely only finitely many $(v,h)\in\xi$ satisfying
$$
\pow (w,(v,h))\leq t.
$$
\item[(P3)] With probability $1$ no $d+1$ points $(v_0,h_0),\ldots, (v_d,h_d)$ from $\xi$ lie on the same downward paraboloid of the form
    $$
    \{(v,h)\in \RR^{d-1}\times E:  \|v - w\|^2 + h = t\}
    $$
with  $(w,t)\in\RR^{d-1}\times E$.
\end{itemize}

Then  $\cL(\xi)$ is a random tessellation in $\RR^{d-1}$ and, moreover, $\cL(\xi)$ is normal with probability $1$.
\end{lemma}

Let $\xi$ be a point process satisfying properties (P1) --- (P3) and let $\cL^*(\xi)$ denote the dual tessellation of $\cL(\xi)$. This tessellation arises from $\cL(\xi)$ by including for distinct points $x_1=(v_1,h_1),\ldots,x_d=(v_d,h_d)$ the simplex $\conv(v_1,\ldots,v_d)$ in $\cL^*(\xi)$ if and only if the Laguerre cells corresponding to $v_1,\ldots,v_d$ all have non-empty interior and share a common point. It follows from Lemma \ref{lm:LTessel} and \cite[Proposition 2]{Sch93} that $\cL^*(\xi)$ is a random simplicial tessellation and, moreover, $\cL^*(\xi)$ is a Laguerre tessellation of the random set
\begin{equation}\label{eq:ApexProcess}
\xi^*:=\left\{(z,-K_{z})\in\RR^{d-1}\times\RR\colon z\in \cF_0(\cL(\xi))\right\},
\end{equation}
where $\cF_0(\cL(\xi))$ is a set of vertices of the Laguerre tessellation $\cL(\xi)$ and for each $z\in\cF_0(\cL(\xi))$ the constant $K_z$ is such that $z\in \cF_0(C((v,h),\xi))$ if and only if $\pow(z, (v,h))=K_{z}$ and there is no $(v,h)\in\xi$ with $\pow(z, (v,h))<K_{z}$.

\section{Construction of $\beta$-, $\beta'$- and Gaussian Delaunay tessellations}\label{sec:Constructions}

\subsection{Paraboloid growth and paraboloid hull processes}\label{sec:ParabHullProc}

In this Section \ref{sec:ParabHullProc} we introduce a paraboloid growth process with overlaps (or, simply, paraboloid growth process) and a paraboloid hull process. They represent a powerful tool which we will use for constructing $\beta$-, $\beta'$- and Gaussian-Delaunay tessellations and investigating their properties. The paraboloid growth and hull processes were first introduced in \cite{SY08, CSY13} in order to study the asymptotic geometry of the Voronoi flower and the convex hull of Poisson point processes in the unit ball, respectively.

Let $\Pi_+$ and $\Pi_{-}$ be the upward and downward standard paraboloids
$$
\Pi_{\pm}:=\{(v,h)\in\RR^{d-1}\times\RR\colon h=\pm\|v\|^2\}.
$$
Let $\Pi_{\pm,x}$ be the translate of $\Pi_{\pm}$, shifted by a vector $x= (v,h)$, that is,
\[
\Pi_{\pm,x}:=\{(v',h')\in\RR^{d-1}\times\RR\colon h'=\pm\|v'-v\|^2+h\}.
\]
We denote the apex of the paraboloid $\Pi_{\pm,x}$ by $\apex\Pi_{\pm,x}:=x$. Further, given a set $A\subset \RR^d$ we put
\begin{align*}
A^{\downarrow}:&=\{(v,h')\in\RR^{d-1}\times\RR\colon (v,h) \in A \text{ for some } h\ge h'\},\\
A^{\uparrow}:&=\{(v,h')\in\RR^{d-1}\times\RR\colon (v,h) \in A \text{ for some } h\leq h'\}.
\end{align*}

Following the definition from \cite{CSY13} for a given locally finite point set $X\subset\RR^d$ we introduce the paraboloid growth model $\Psi(X)$:
$$
\Psi(X):=\bigcup\limits_{x\in X}\Pi^{\uparrow}_{+,x}.
$$
In particular, given a Poisson point process $\xi$, $\Psi(\xi)$ is called the \textbf{generalized paraboloid growth process with overlaps} or simply the  \textbf{paraboloid growth process}. In other words, $\Psi(\xi)$ is a Boolean model (germ-grain model) in $\RR^d$ with paraboloid grains $\Pi_{+}$, see \cite[Chapter 9]{SW}. The name \textit{generalized paraboloid growth process with overlaps} comes from the original interpretation of this construction \cite{SY08} as the graph of the growth of crystals, when each crystal is born at the spatial coordinate $v$ at corresponding time $h$, where $x:=(v,h)\in\xi$, and growths with the same speed in all directions. Unlike the traditional growth scheme, where crystals stop growing in the direction if they touch each other, we allow them to overlap in our model. It means that a new crystal can be born at coordinate $v$ at time $h$ and starts growing, even if at that moment the point $v$ has already been occupied by another crystal. This leads to the definition of an extreme crystal as a crystal which at some time is not fully  covered by other crystals or, more formally, a point $x\in\xi$ is called \textbf{extreme} in the paraboloid growth process $\Psi(\xi)$ if and only if its associated paraboloid is not fully covered by the paraboloids associated with other points of $\xi$, i.e., if
$$
\Pi^{\uparrow}_{+,x}\not\subset \bigcup\limits_{y\in \xi, y\neq x}\Pi^{\uparrow}_{+,y}.
$$
We denote by $\ext(\Psi(\xi))$ the set of all extreme points of the paraboloid growth process $\Psi(\xi)$.

Further, we introduce the paraboloid hull process, which can be regarded as a dual to the paraboloid growth process and was considered in \cite{CSY13}, see Figure \ref{fig:GrowthHullProcess} for an illustration. The paraboloid hull process is designed in order to exhibit properties analogous to those of convex polytopes, where paraboloids are playing a role of hyperplanes, spatial coordinates $v$ are playing a role of spherical coordinates and hight coordinates $h$ are playing a role of radial coordinates.

For any collection $x_1:=(v_1,h_1),\ldots,x_k:=(v_k,h_k)$ of $1\leq k\leq d$ points in $\RR^{d-1}\times\RR$ with affinely independent spatial coordinates $v_1,\ldots,v_k\in  \RR^{d-1}$, we define $\Pi(x_1,\ldots,x_k)$ to be the intersection of $\aff(v_1,\ldots,v_k)\times\RR$ with any translate of $\Pi_{-}$ containing all the points $x_1,\ldots,x_k$. It should be mentioned here that although such translates of $\Pi_{-}$ are not unique for $k<d$, their intersections with $\aff(v_1,\ldots,v_k)\times\RR$ all coincide so that $\Pi(x_1,\ldots,x_k)$ is uniquely determined. Further we define the parabolic face $\Pi[x_1,\ldots,x_k]$ as
\[
\Pi[x_1,\ldots,x_k]:=\Pi(x_1,\ldots,x_k)\cap \left(\conv(v_1,\ldots,v_k)\times\RR\right).
\]
Following \cite{CSY13} we say that a set $A\subset \RR^d$ is upwards paraboloid convex if for each $x_1=(v_1,h_2),x_2=(v_2,h_2)\in A$ with $v_1\neq v_2$ we have $\Pi[x_1,x_2]\subset A$ and if for each $x=(v,h)\in A$ we have $x^{\uparrow}:=\{x\}^\uparrow\subset A$.

Finally, given a locally finite point set $X\subset\RR^d$ we define its paraboloid hull $\Phi(X)$ to be the smallest upwards paraboloid convex set containing $X$. Given a Poisson point process $\xi$, we define the  \textbf{paraboloid hull process} as $\Phi(\xi)$. Following the arguments analogous to \cite[Lemma 3.1]{CSY13} one has that, with probability one,
\begin{equation}\label{eq_9.11_1}
\Phi(\xi)=\bigcup\limits_{(x_1,\ldots,x_d)\in\xi_{\neq}^d}\left(\Pi[x_1,\ldots, x_d]\right)^{\uparrow},
\end{equation}
where $\xi_{\neq}^d$ is the collection of all $d$-tuples of distinct points of $\xi$, see Figure \ref{fig:ParabolicProcess} for an illustration. We shell point out that with probability one any $d$ distinct points $(x_1,\ldots,x_d)\in\xi_{\neq}^d$ have affinely independent spatial coordinates and the right-hand side in \eqref{eq_9.11_1} is almost surely well defined.

\begin{figure}[t]
\centering
\includegraphics[width=0.8\columnwidth]{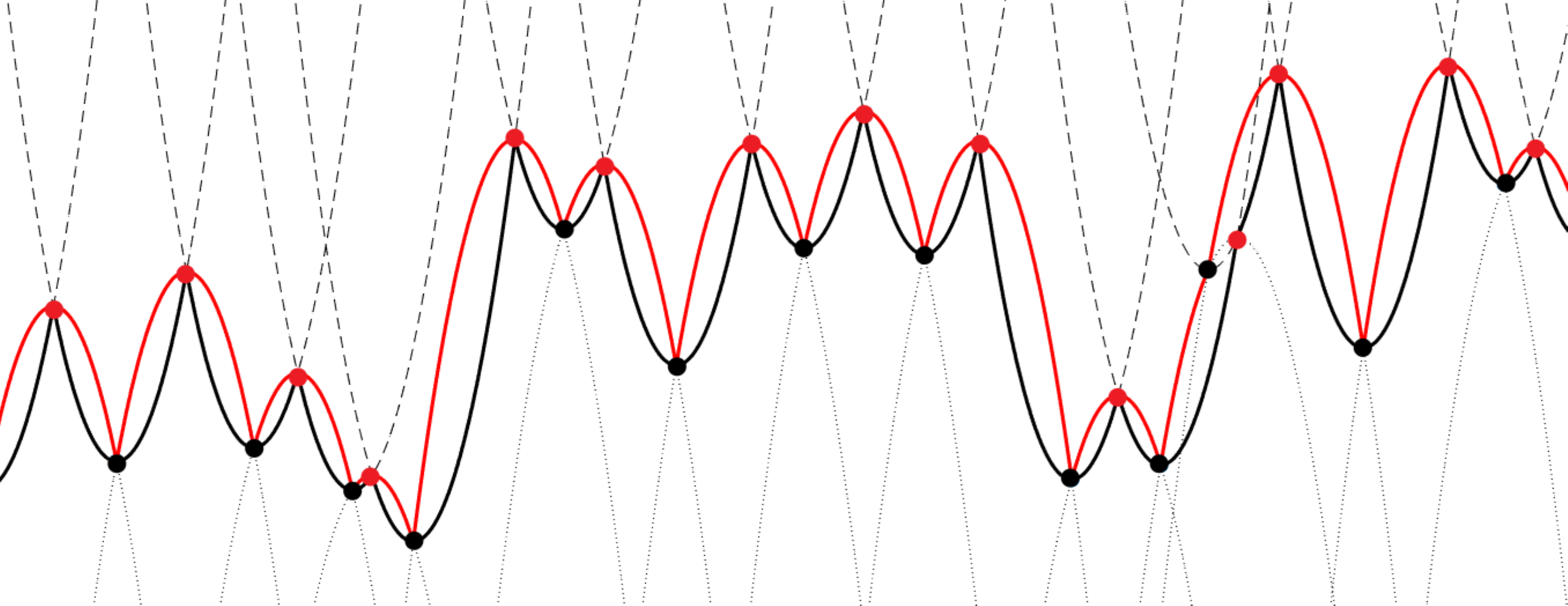}
\caption{Paraboloid growth process (black) with corresponding paraboloid hull process (red).}
\label{fig:GrowthHullProcess}
\end{figure}

For $(x_1,\ldots,x_d)\in\xi_{\neq}^d$ with affinely independent spatial coordinates $v_1,\ldots, v_d$ the set $\Pi[x_1,\ldots,x_d]$ is called a paraboloid sub-facet of $\Phi(\xi)$ if $\Pi[x_1,\ldots,x_d] \in\partial \Phi(\xi)$. Two different paraboloid sub-facets $\Pi[x_1,\ldots,x_d]$ and $\Pi[y_1,\ldots,y_d]$ are called co-paraboloid provided that $\Pi(x_1,\ldots,x_d)=\Pi(y_1,\ldots,y_d)$. Note, that although all points in $\{x_1,\ldots, x_d\}$ and $\{y_1,\ldots,y_d\}$ must be distinct according to the definition of a paraboloid sub-facet, the sets $\{x_1,\ldots,x_d\}$ and $\{y_1,\ldots,y_d\}$ themselves do not need to be disjoint. The definition of co-paraboloid sub-facets is equivalent to saying that all points $y_1,\ldots,y_d,x_1\ldots,x_d$ are lying on the same shift of standard downward paraboloid $\Pi_{-}$. By a \textbf{paraboloid facet (or $(d-1)$-dimensional paraboloid face)} of $\Phi(\xi)$ we understand the union of each maximal collection of co-paraboloid sub-facets. If $\xi$ is a Poisson point process satisfying (P3) each paraboloid facet of $\Phi(\xi)$ with probability one consists of exactly one sub-facet. We can thus say that with probability one $\Pi[x_1,\ldots,x_d]$ is a paraboloid facet of $\Phi(\xi)$ if and only if $\xi \cap \Pi(x_1,\ldots,x_d)^{\downarrow}=\{x_1,\ldots, x_d\}$. The points $x_1,\ldots, x_d$ are the \textbf{vertices} of the paraboloid facet $\Pi[x_1,\ldots,x_d]$. More formally, the vertex is a zero dimensional face of $\Phi(\xi)$, which arises as an intersection of suitable $d$-tuple of adjacent $(d-1)$-dimensional faces of  $\Phi(\xi)$. The collection of all vertices of $\Phi(\xi)$ is denoted by $\ver(\Phi(\xi))$ and it is easy to see that $\ver(\Phi(\xi))\subset \xi$.

The following equality builds a link between the paraboloid growth and paraboloid hull processes, see \cite[Equation (3.17)]{CSY13} and Figure \ref{fig:GrowthHullProcess}:
$$
\ver(\Phi(\xi))=\ext(\Psi(\xi)).
$$

Using the paraboloid growth and paraboloid hull processes we can construct random diagrams in $\RR^{d-1}$ which under suitable conditions on the underlying point process $\xi$ give rise to random tessellations. Consider first the paraboloid growth process $\Psi(\xi)$. Given a point $x=(v,h)\in\ext(\Psi(\xi))$ define the $\Psi$-cell of $x$ as
$$
C_{\Psi}(x,\xi):=\{w\in \RR^{d-1}\colon w^{\uparrow}\cap \partial \Psi(\xi) \in \Pi_{+,x}\}.
$$
In other words $w$ belongs to $C_{\Psi}(x,\xi)$ if and only if $\|w-v\|^2+h\leq \|w-v'\|+h'$ for any $(v',h')\in\xi$. Thus, the $\Psi$-cell of an extreme point $x$ of the paraboloid growth processes $\Psi(\xi)$ has non-empty interior and coincides with the Laguerre cell $C(x, \xi)$, which we defined in Section \ref{sec:Laguerre_tess}. Next, we construct the diagram $\cM_{\Psi}(\xi)$ as the collection of all $\Psi$-cells:
\begin{equation}\label{eq:DefTPsi}
\cM_{\Psi}(\xi):=\{C_{\Psi}(x,\xi)\colon x\in\ext(\Psi(\xi))\}.
\end{equation}
We directly have that $\cM_{\Psi}(\xi)\subset \cL(\xi)$. On the other hand, if the Laguerre cell $C(x, \xi)$ has non-empty interior then $x\in \ext(\Psi(\xi))$ and thus $\cM_{\Psi}(\xi)=\cL(\xi)$.

Now, consider  the paraboloid hull processes $\Phi(\xi)$ and construct a diagram $\cM_{\Phi}(\xi)$ in the following way: for any paraboloid facet $f$ of $\Phi(\xi)$ with vertices $x_1,\ldots, x_m\in\xi$ a polytope $\conv(v_1,\ldots,v_m)$ belongs to $\cM_{\Phi}(\xi)$. As already mentioned, if the the point process $\xi$ satisfies condition (P3), then with probability one any paraboloid facet of $\Phi(\xi)$ has exactly $d$ vertices and the construction above can be alternatively described as follows: for any collection $x_1=(v_1,h_1),\ldots,x_d=(v_d,h_d)$ of pairwise distinct points from $\xi$ we say that the simplex $\conv(v_1,\ldots,v_d)$ belongs to $\cM_{\Phi}(\xi)$ if and only if $\Pi[x_1,\ldots,x_d]$ is a paraboloid facet of $\Phi(\xi)$.

If the point process $\xi$ satisfies conditions (P1) --- (P3) of Lemma \ref{lm:LTessel} then $\cM_{\Phi}(\xi)$ and $\cM_{\Psi}(\xi)$ are a random tessellations in $\RR^{d-1}$ and $\cM_{\Phi}(\xi)=\cL^{*}(\xi)$ almost surely.

\begin{figure}[t]
\centering
\includegraphics[width=\columnwidth]{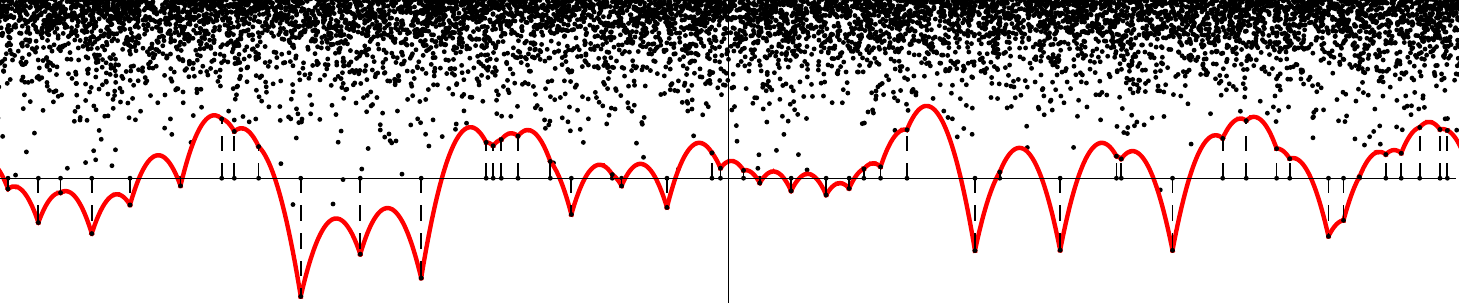}
\caption{Construction of the Gaussian-Delaunay tessellation for $d=2$. The figure shows the Poisson point process $\zeta$ introduced in Section \ref{subsec:GaussianDelaunay} and the corresponding paraboloid hull $\Phi(\zeta)$.}
\label{fig:ParabolicProcess}
\end{figure}

\subsection{$\beta$-Delaunay and $\beta^{'}$-Delaunay tessellations}\label{subsec:betaDelaunay}

The $\beta$- and $\beta'$-Delaunay tessellations have been introduced in part I \cite{Part1} of this paper, where two alternative definitions using the concept of Laguerre tessellations and the notion of paraboloid hull processes were given. The first approach is more suitable for verifying properties of the tessellation and defining the typical cell, while the second approach is used for investigating the probabilistic behaviour of the typical cell. For convenience of the reader and to make the paper more self-contained will briefly recall both definitions here.

As an underlying point process for the $\beta$-Delaunay tessellation we consider a Poisson point process $\eta_{\beta} = \eta_{\beta,\gamma}$ in $\RR^{d-1}\times \RR_{+}$ with $\beta >-1$ and intensity measure having density
\begin{equation}\label{eq:BetaPoissonIntensity}
(v,h)\mapsto \gamma\,c_{d,\beta}h^{\beta},\qquad c_{d,\beta}:={\Gamma\left({d\over 2}+\beta+1\right)\over \pi^{d\over 2}\Gamma(\beta+1)},\, \gamma > 0,
\end{equation}
with respect to the Lebesgue measure on $\RR^{d-1}\times \RR_{+}$. Analogously for the $\beta'$-Delaunay tessellation we consider a Poisson point process $\eta^{\prime}_{\beta}=\eta^{\prime}_{\beta,\gamma}$ in $\RR^{d-1}\times\RR_{-}^*$ with $\beta >(d+1)/2$ and intensity measure having density
\begin{equation}\label{eq:BetaPrimePoissonIntensity}
(v,h)\mapsto\gamma\,c'_{d,\beta}\,(-h)^{-\beta},\qquad c'_{d,\beta}:={\Gamma\left(\beta\right)\over \pi^{d\over 2}\Gamma(\beta-{d\over 2})},\, \gamma > 0,
\end{equation}
with respect to the Lebesgue measure on $\RR^{d-1}\times\RR_{-}^*$.

It is straightforward to see (compare with \cite[Lemma 3.3]{Part1}) that the point processes $\eta_{\beta}$ and $\eta^{\prime}_{\beta}$ satisfy conditions (P1) --- (P3) of Lemma \ref{lm:LTessel} and, thus, $\cL(\eta_\beta)$ and $\cL(\eta^{\prime}_\beta)$ are random tessellations in $\RR^{d-1}$, which are almost surely normal. The dual tessellations $\cL^*(\eta_\beta)$ and $\cL^*(\eta^{\prime}_\beta)$ are with probability one random simplicial tessellations in $\RR^{d-1}$. The random tessellation $\cD_\beta:=\cL^*(\eta_{\beta})$ is called the \textbf{$\beta$-Delaunay tessellation} in $\RR^{d-1}$ and the random tessellation $\cD^{\prime}_\beta:=\cL^*(\eta^{\prime}_{\beta})$ is called the \textbf{$\beta'$-Delaunay tessellation} in $\RR^{d-1}$.

Alternatively, consider first the paraboloid hull processes $\Phi(\eta_{\beta})$ and $\Phi(\eta^{\prime}_{\beta})$, and construct the diagrams $\cM_{\Phi}(\eta_{\beta})$ and $\cM_{\Phi}(\eta^{\prime}_{\beta})$ as described in Section \ref{sec:ParabHullProc}. Since the Poisson point processes $\eta_{\beta}$ and $\eta^{\prime}_{\beta}$  satisfy conditions (P1) --- (P3) of Lemma \ref{lm:LTessel}, the tessellations $\cM_{\Phi}(\eta_{\beta})$ and $\cM_{\Phi}(\eta^{\prime}_{\beta})$ coincide almost surely with $\beta$- and $\beta'$-Delaunay tessellations, respectively, as defined above.

\subsection{The Gaussian-Delaunay tessellation}\label{subsec:GaussianDelaunay}

\begin{figure}[t]
	\centering
	\begin{minipage}[b]{0.45\linewidth}
		\includegraphics[width=0.9\columnwidth]{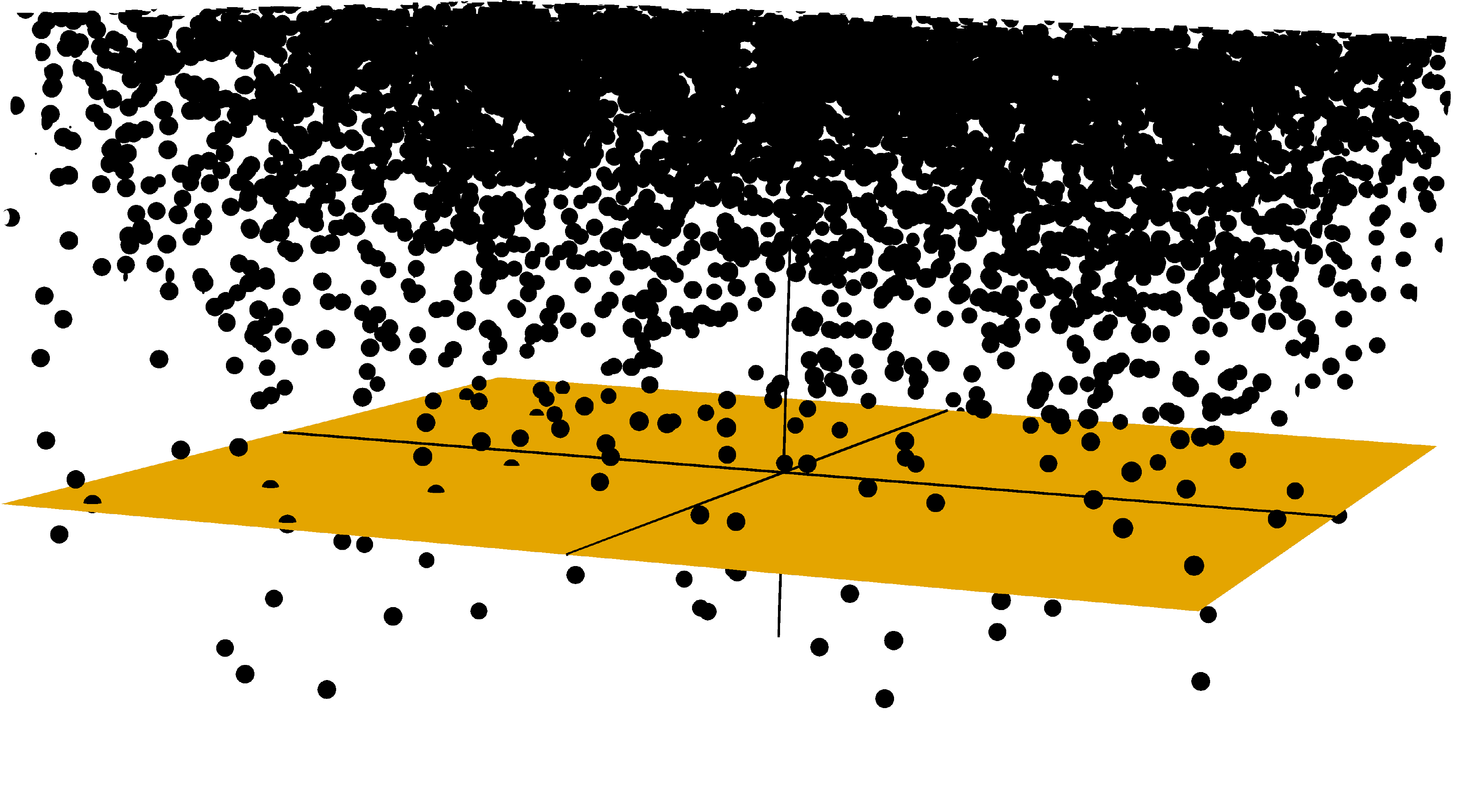}
		\caption{Poisson point processes in $\RR^d$ with $d=3$ used to construct the Gaussian-Delaunay tessellations in the plane. The plane $h=0$ in which the tessellation is constructed is shown in yellow.}
		\label{fig:d=3}
	\end{minipage}
	\quad
	\begin{minipage}[b]{0.45\linewidth}
	\includegraphics[width=0.9\columnwidth]{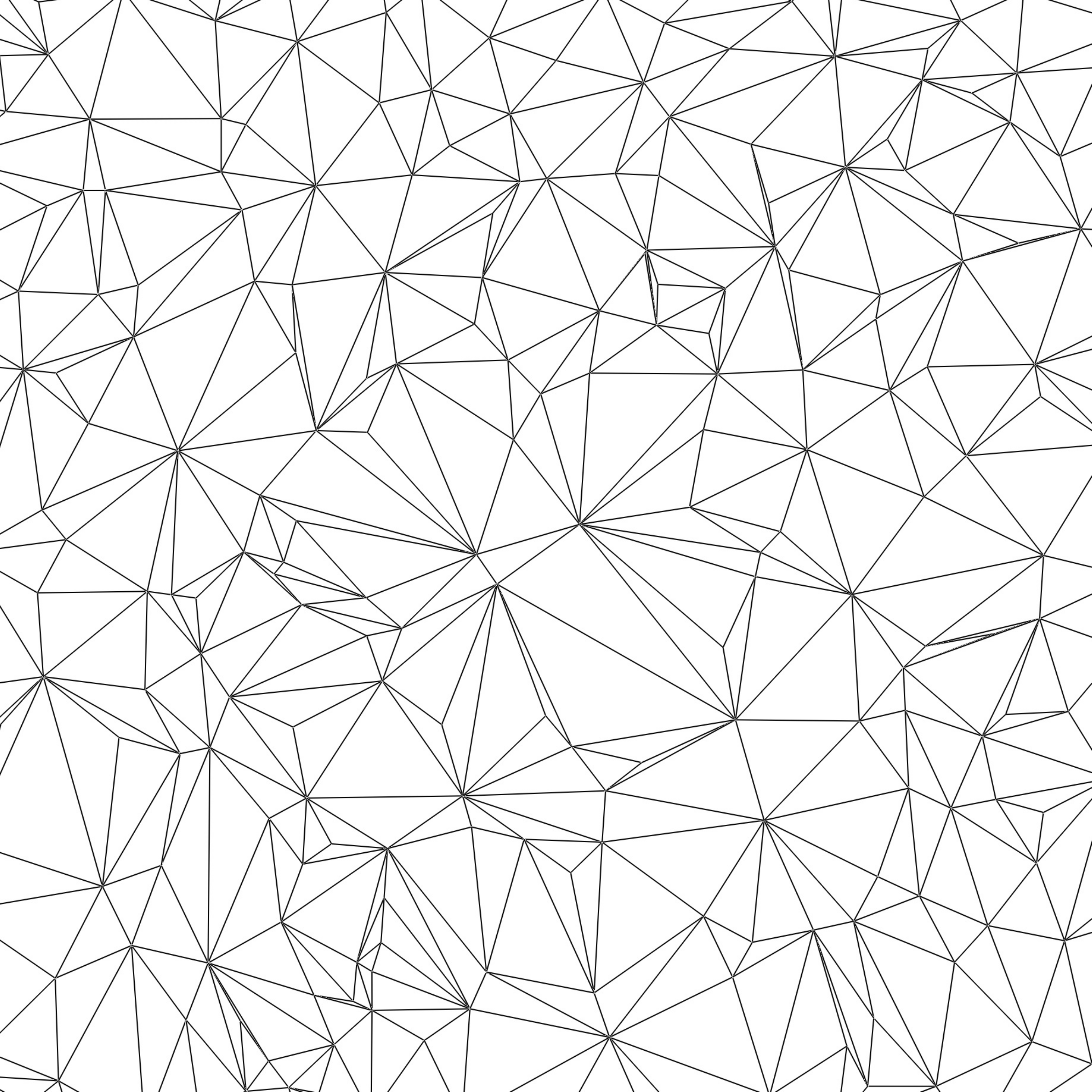}
	\caption{Realization of a Gaussian-Delaunay tessellation in $\RR^2$ illustrated in the window $[-13,13]^2$.}
	\label{fig:gaussian}
\end{minipage}
	
\end{figure}

Following the construction of $\beta$- and $\beta^{\prime}$-Delaunay tessellations let us now define what we mean by the Gaussian-Delaunay tessellation. It is based on a Poisson point process $\zeta$ in $\RR^{d-1}\times\RR$ whose intensity measure has density
$$
(v,h)\mapsto {1\over (2\pi)^{d/2}}\,e^{h/2},
$$
with respect to the Lebesgue measure on $\RR^{d-1}\times\RR$.

As in the previous subsection we will consider two alternative definitions of the resulting tessellation. First of all, we need to ensure that properties (P1) --- (P3) from Lemma \ref{lm:LTessel} hold for the point process $\zeta$. Property (P1) holds because the projection of the Poisson point process $\zeta$ to the space component $\RR^{d-1}$ is with probability one a dense subset of $\RR^{d-1}$.
Property (P3) also holds for any Poisson point process in $\RR^d$ whose intensity measure is absolutely continues with respect to the Lebesgue measure, see the proof of \cite[Lemma 3.3]{Part1}. In order to check (P2) it is enough to show that
$$
\EE \sum_{(v,h)\in\zeta}{\bf 1}(\|v-w\|^2+h\leq t)<\infty.
$$

Indeed, applying Mecke's formula for Poisson point processes \cite[Theorem 3.2.5]{SW} we obtain
\begin{align*}
\EE \sum_{(v,h)\in\zeta}{\bf 1}(\|v-w\|^2-h\leq t)&= {1\over (2\pi)^{d/2}}\int_{\RR^{d-1}}\int_{\RR}{\bf 1}(\|v-w\|^2+h\leq t)e^{h/2}\dd h\,\dd v\\
&= {1\over (2\pi)^{d/2}}\int_{\RR^{d-1}}\int_{-\infty}^{t-\|v\|^2}e^{h/2}\dd h\,\dd v\\
&= {2\over (2\pi)^{d/2}}e^{t/2}\int_{\RR^{d-1}}e^{-\|v\|^2/2}\,\dd v\\
&={2(d-1)\kappa_{d-1}\over (2\pi)^{d/2}}e^{t/2}\int_0^\infty e^{-u^2/2}\,u^{d-2}\,\dint u\\
& ={\sqrt{2}\over \sqrt{\pi}}e^{t/2}<\infty.
\end{align*}
Hence, the Laguerre diagram $\cL(\zeta)$ based on $\zeta$ is an almost surely normal random tessellation in $\RR^{d-1}$. The dual tessellations $\cD:=\cL^*(\zeta)$ is with probability one a simplicial random tessellation and is called the \textbf{Gaussian-Delaunay tessellation} in $\RR^{d-1}$.

\begin{figure}[!t]
\centering
\includegraphics[width=0.31\textwidth]{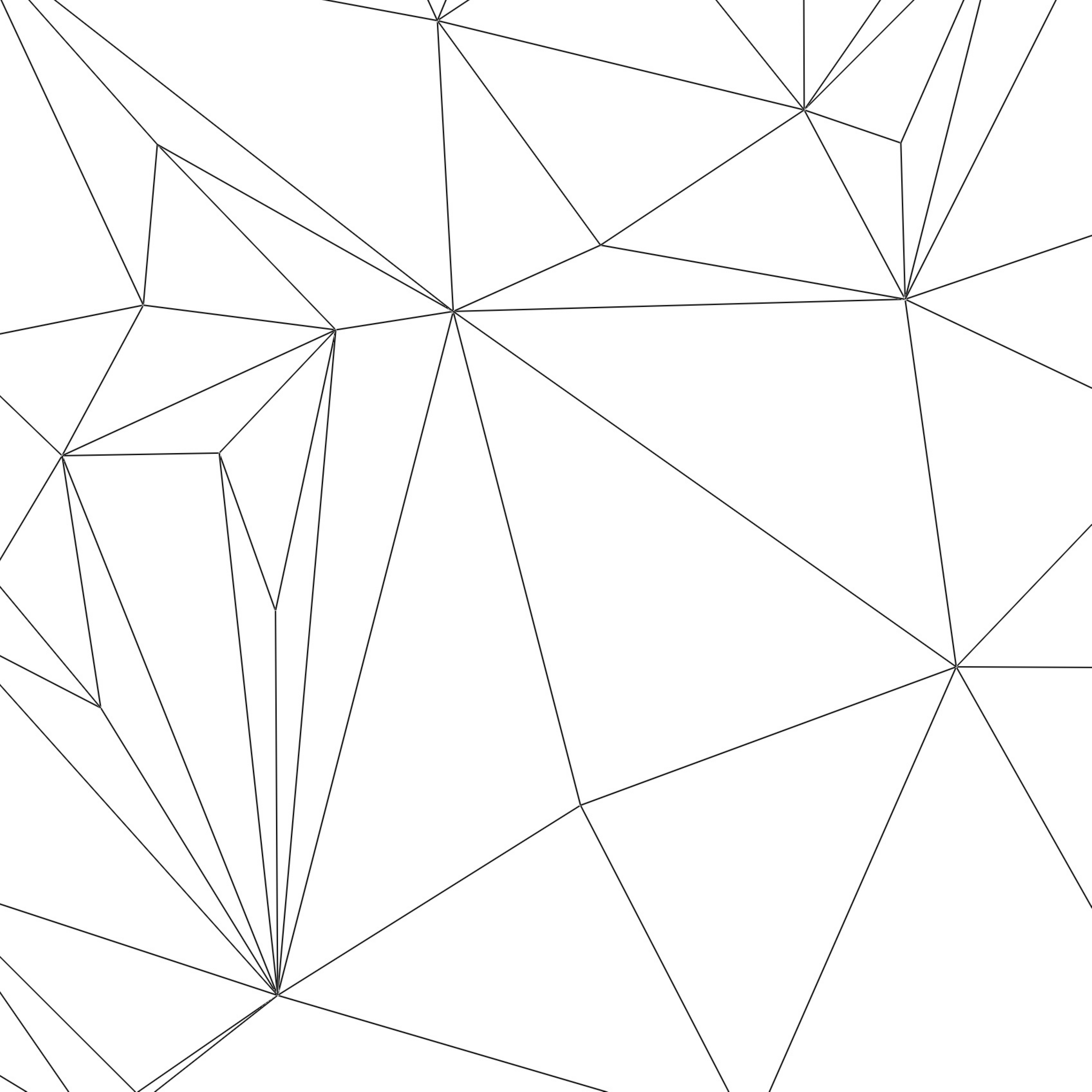}\hspace{0.01\textwidth}
\includegraphics[width=0.31\textwidth]{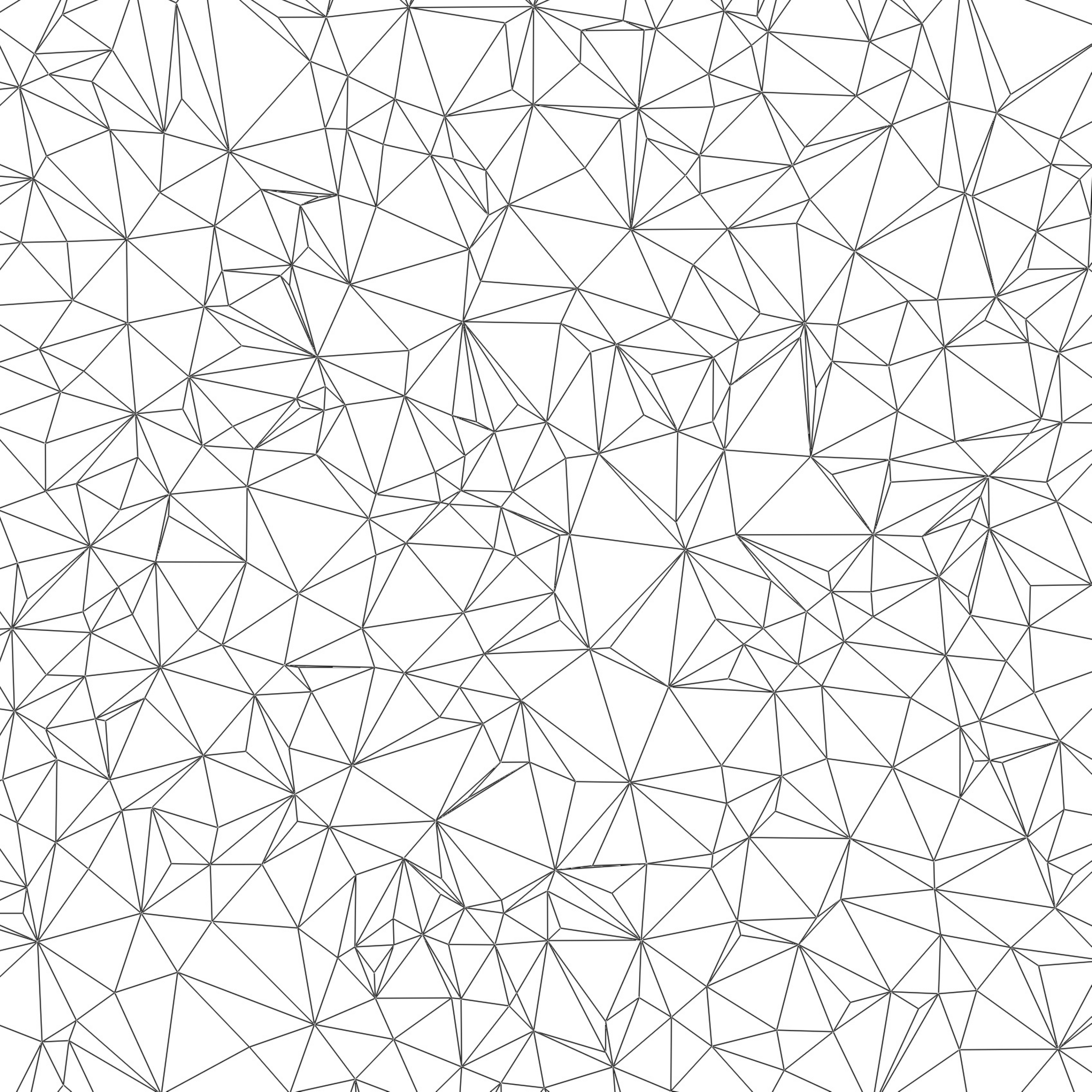}\hspace{0.01\textwidth}
\includegraphics[width=0.31\textwidth]{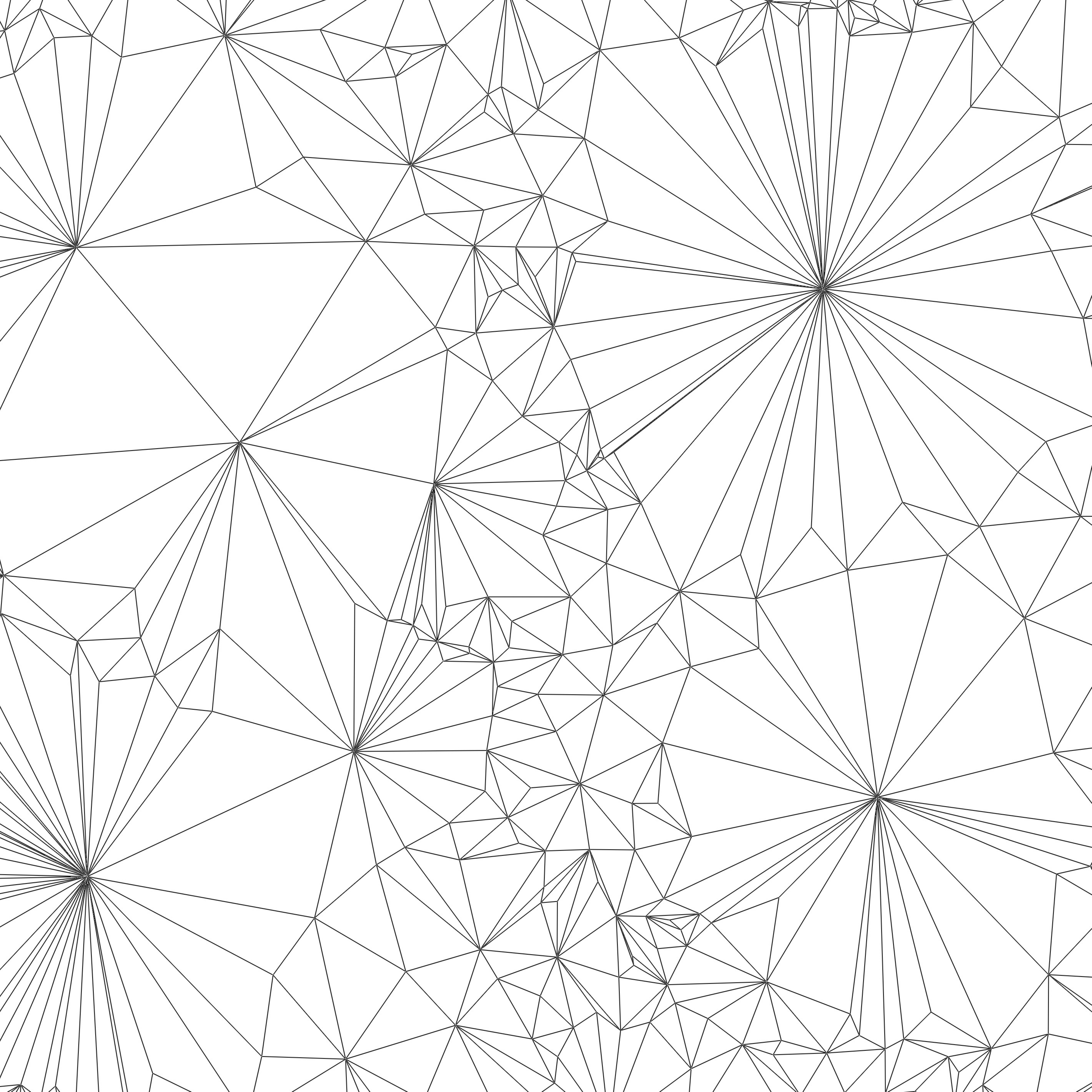}
\caption{Realizations of Delaunay tessellations in $\RR^2$ in the window $[-4,4]^2$. Left: The Gaussian-Delaunay tessellation. Middle: The $\beta$-Delaunay tessellation with $\beta = 10$. Right: The $\beta^{\prime}$-Delaunay tessellation with $\beta = 2.5$. }
\label{fig:tessellations}
\end{figure}

Alternatively, based on the paraboloid hull processes $\Phi(\zeta)$ one can construct the diagram $\cM_{\Phi}(\zeta)$ as described in Section \ref{sec:ParabHullProc} and, as before, conclude that with probability one $\cM_{\Phi}(\zeta)$ coincides with Gaussian-Delaunay tessellation $\cL^*(\zeta)=\cD$.

We emphasize at this point that the $\beta$-, $\beta'$- and the Gaussian-Delaunay tessellation in $\RR^{d-1}$ are in fact stationary random tessellations, meaning that their laws are invariant under shits in $\RR^{d-1}$. The reason for this is simply that the Poisson point processes $\eta_{\beta}$, $\eta_{\beta}'$ and $\zeta$ underlying the construction of these tessellations are stationary under shifts of their spatial components.

\section{Convergence to the Gaussian limiting tessellation}\label{sec:Convergence}

\subsection{Convergence of the point processes}\label{subsec:ConvergencePP}

Let us recall that the $\beta$-Delaunay and $\beta^{\prime}$-Delaunay  tessellations were build on the Poisson point processes $\eta_{\beta,\gamma}$ in $\RR^{d-1}\times \RR_{+}$ and $\eta^{\prime}_{\beta,\gamma}$ in $\RR^{d-1}\times \RR^*_{-}$, whose intensity measures have densities
\begin{align*}
\gamma c_{d,\beta}\,&h^{\beta},\qquad\qquad\, (v,h)\in\RR^{d-1}\times\RR_{+},\, \gamma >0,\\
\gamma c^{\prime}_{d,\beta}\,&(- h)^{-\beta},\qquad (v,h)\in\RR^{d-1}\times\RR_{-}^*,\, \gamma >0,
\end{align*}
with respect to the Lebesgue measure, respectively. For $\beta>0$ consider the transformations
\begin{align*}
Q_\beta &: \RR^{d-1}\times\RR_{+} \mapsto \RR^{d-1}\times\RR,(v,h)\mapsto \big((2\beta)^{1/2}\,v\,,2\beta(h-1)\big),\\
Q^{\prime}_\beta &: \RR^{d-1}\times\RR_{-}^* \mapsto \RR^{d-1}\times\RR,(v,h)\mapsto \big((2\beta)^{1/2}\,v\,,2\beta(h+1)\big).
\end{align*}
We are interested in the rescaled point processes
$$
\zeta_\beta:=Q_\beta(\eta_{\beta,\sqrt{2\beta}})\qquad\text{and}\qquad\zeta^{\prime}_\beta:=Q^{\prime}_\beta(\eta^{\prime}_{\beta,\sqrt{2\beta}}),
$$
where we choose the intensity $\gamma$ as $\sqrt{2\beta}$. For any fixed $C>0$ and sufficiently large $\beta$,  these can be considered as point processes on $\RR^{d-1} \times [-C,C]$. Unfortunately, it is not possible to consider them as point processes on $\RR^{d-1} \times \RR$ because the corresponding point configurations are not locally finite.

We will start by proving the following lemma, which shows that both $\zeta_\beta$ and $\zeta_\beta'$ converge in distribution (with respect to the topology of vague convergence on the space of locally finite counting measures on $\RR^{d-1}\times [-C,C]$) to the Gaussian limiting point process $\zeta$, as $\beta\to\infty$, which we defined in Section \ref{subsec:GaussianDelaunay}.  Formally, this means that for any finite collection of relatively compact Borel sets $B_1,\ldots,B_n\subset\RR^{d-1}\times\RR$ with $\zeta(\partial B_i)=0$ for $i\in\{1,\ldots,n\}$ one has that the $n$-dimensional random vector $(\zeta_{\beta}(B_1),\ldots,\zeta_{\beta}(B_n))$ converges in distribution to the $n$-dimensional random vector $(\zeta(B_1),\ldots,\zeta(B_n))$, as $\beta\to\infty$ (and similarly with $\zeta_{\beta}$ replaced by $\zeta_{\beta}^{'}$ when the $\beta'$-Delaunay tessellation is considered). We indicate this by writing $\zeta_\beta\overset{d}{\longrightarrow}\zeta$, as $\beta\to\infty$. We refer to \cite[Chapter 16]{Kallenberg} for background material on the convergence of point processes.

\begin{lemma}\label{lm:processconverge}
For every $C>0$
one has that $\zeta_\beta\overset{d}{\longrightarrow}\zeta$ and $\zeta^{\prime}_\beta\overset{d}{\longrightarrow}\zeta$, as $\beta\to\infty$, where all point processes are viewed as random elements with values in the space of locally finite counting measures on $\RR^{d-1}\times [-C,C]$.
\end{lemma}
\begin{proof}
We will only prove the result for $\zeta_\beta$, the convergence for $\zeta^{\prime}_\beta$ can be shown in the same way. By the mapping property of Poisson point processes, $\zeta_\beta$ is a Poisson point process on $\RR^{d-1}\times [-C,C]$, if $\beta$ is sufficiently large,  with intensity measure satisfying
\begin{align*}
\EE\zeta_\beta(A\times B) &= {\sqrt{2\beta}\,c_{d,\beta}}\int_{\RR^{d-1}}\int_{\RR}{\bf 1}\big((2\beta)^{1/2}\,v\in A,2\beta(h-1)\in B, h\geq 0\big)\,h^{\beta}\,\dint h\dint v
\end{align*}
for all bounded Borel sets $A\subset\RR^{d-1}$ and $B\subset [-C,C]$. Substituting $w=\sqrt{2\beta}\,v$ and $s=2\beta(h-1)$ this transforms as follows:
\begin{align}
\EE\zeta_\beta(A\times B) &= {c_{d,\beta}\over(2\beta)^{d/2}}\int_{\RR^{d-1}}\int_{\RR}{\bf 1}\Big(w\in A,s\in B,1+{s\over 2\beta}\geq 0\Big)\,\Big(1+{s\over 2\beta}\Big)^{\beta}\,\dint s\dint w.\label{eq:28.01_2}
\end{align}
Since, by Stirling's formula, $c_{d,\beta}$ is asymptotically equivalent to $(\beta/\pi)^{d/2}$, as $\beta\to\infty$, we conclude from the dominated convergence theorem that
$$
\lim_{\beta\to\infty}\EE\zeta_\beta(A\times B)  = {1\over (2\pi)^{d/2}}\int_{\RR^{d-1}}\int_{\RR}{\bf 1}(w\in A,s\in B)\,e^{s/2}\,\dint s\dint w.
$$
This is equivalent to say that
$$
\lim_{\beta\to\infty}\EE[\zeta_\beta(A\times B)]=\EE[\zeta(A\times B)]
$$
for all relatively compact Borel sets $A\subset\RR^{d-1}$ and $B\subset [-C,C]$. However, since $\zeta_\beta(A\times B)$ and $\zeta(A\times B)$ are Poisson random variables, this readily implies the convergence in distribution $\zeta_\beta(A\times B)\overset{d}{\longrightarrow}\zeta(A\times B)$, as $\beta\to\infty$ (for example by considering the characteristic functions). By \cite[Theorem 16.16 (iv)]{Kallenberg} this in turn yields the distributional convergence of the point process $\zeta_\beta$ to $\zeta$, as $\beta\to\infty$.
\end{proof}

\subsection{Convergence of the tessellations}\label{subsec:ConvergenceTess}

Motivated by Lemma \ref{lm:processconverge}, it is natural to expect that a similar distributional convergence holds on the level of tessellations as well. However, it is not straightforward to introduce a suitable notion of weak convergence of random tessellations. One reason behind is that there is no natural topological structure on the space of tessellations and it is not clear whether the space of tessellations in $\RR^{d-1}$ is even Polish or not (see the discussion in \cite[Section 2.3]{HugThaeleSplitting}). A possible way out is to view a random tessellation as a random closed set in the common sense of stochastic geometry and to work with the weak convergence of random closed sets with respect to the so-called Fell topology on the space $\cC$, the space of all compact subsets of $\RR^{d-1}$, see \cite[Chapter 2]{SW} for a general introduction to the theory of random closed set. In our case this turns out to be possible since for any $\beta>-1$ (or $\beta>(d+1)/2$) the $\beta$- (or $\beta'$-)Delaunay tessellation and also their common limit, the Gaussian-Delaunay tessellation, are in fact tessellations of $\RR^{d-1}$.

To formalize this idea we need to introduce some notation. Given a stationary random tessellation $\cM$ in $\RR^{d-1}$ we denote by
$$
\sk{M}= \skel(\cM) := \bigcup_{c\in \cM}\partial c,
$$
the skeleton of $\cM$, see \cite[Definition 10.1.4]{SW}. By the discussion in \cite[p.\ 464]{SW} it follows that $\sk{M}$ is a stationary random closed set in $\RR^{d-1}$. We recall that the capacity functional $T_{Z}(C)$ of a general random closed set $Z\subset\RR^{d-1}$ is defined as
$$
T_{Z}(C)=\PP(Z\cap C \neq \varnothing),\qquad C\in\cC.
$$

Let us point out that the Poisson point processes $\zeta_{\beta}$, $\zeta^{\prime}_{\beta}$ and $\zeta$ satisfy properties (P1) --- (P3) stated in Lemma \ref{lm:LTessel}. This allows us to consider the random closed sets
$$
\widetilde{\sk{D}}_{\beta}=\skel(\widetilde{\cD}_{\beta}):=\skel(\cL^*(\zeta_{\beta})),\quad \widetilde{\sk{D}}^{\prime}_{\beta}=\skel(\widetilde{\cD}^{\prime}_{\beta}):=\skel(\cL^*(\zeta^{\prime}_{\beta})),\quad\sk{D}=\skel(\cD):=\skel(\cL^*(\zeta)),
$$
where we use the notation introduced in Sections \ref{sec:Laguerre_tess} and \ref{sec:ParabHullProc}. We also consider the capacity functionals
$$
T_{\beta}(C):=T_{\widetilde{\sk{D}}_{\beta}}(C),\quad T^{\prime}_{\beta}(C):=T_{\widetilde{\sk{D}}^{\prime}_{\beta}}(C)\quad\text{and}\quad T(C):=T_{\sk{D}}(C),\qquad C\in\cC.
$$

Our principal goal in this section is to prove the following theorem, which is the counterpart to Lemma \ref{lm:processconverge} on the level of tessellations. More precisely, we shall prove that, as $\beta\to\infty$, the random closed sets $\widetilde{\sk{D}}_{\beta}$ and $\widetilde{\sk{D}}^{\prime}_{\beta}$ converge weakly (with respect to the \textbf{Fell topology} on the space $\cC$) to the random closed set $\sk{D}$. Formally, this means that the capacity functionals $T_\beta$ and $T_\beta'$ both converge, as $\beta\to\infty$, to the capacity functional $T$ of the Gaussian-Delaunay tessellation. For background material on the convergence of random sets we refer to \cite[Chapter 16]{Kallenberg} or \cite[Chapter 2]{SW} as well as to the explanations we give after the statement of our result.

\begin{theorem}\label{tm:CobvergenceGaussian}
The random closed sets $\widetilde{\sk{D}}_{\beta}$ and $\widetilde{\sk{D}}^{\prime}_{\beta}$ converge weakly to the random closed set $\sk{D}$, as $\beta \to \infty$.
\end{theorem}

\begin{remark}
The theorem above can be reformulated for the $\beta$- and $\beta^{\prime}$-Delaunay tessellations as follows. For some positive $t$ define, analogously for $\beta^{\prime}$-Delaunay tessellation, the rescaled version of the $\beta$-Delaunay tessellation
$$
t\cD_{\beta}:=\{tc\colon c\in\cD_{\beta}\}.
$$
Then the random closed sets $\sqrt{2\beta}\sk{D}_{\beta}:=\skel(\sqrt{2\beta}\,\cD_{\beta})$ and $\sqrt{2\beta}\sk{D}^{\prime}_{\beta}:=\skel(\sqrt{2\beta}\,\cD^{\prime}_{\beta})$ converge weakly to the random closed set $\sk{D}$, as $\beta\to\infty$.
This follows from the fact that $\widetilde{\sk{D}}_{\beta}=\sqrt{2\beta}\sk{D}_{\beta}$ and $\widetilde{\sk{D}}^{\prime}_{\beta}=\sqrt{2\beta}\sk{D}^{\prime}_{\beta}$, and is due to the property of the transformation $Q_{\beta}$, which preserves the parabolic structure of the set. Namely the image of paraboloid $\Pi_{(v,h),\pm}$ with apex $(v,h)$ is the paraboloid $\Pi_{(v',h'),\pm}$ with apex $(v',h'):=(\sqrt{2\beta}v, 2\beta (h-1))$.
\end{remark}

The proof of Theorem \ref{tm:CobvergenceGaussian} is based on a careful analysis of the capacity functionals of $\widetilde{\sk{D}}_{\beta}$, $\widetilde{\sk{D}}^{\prime}_{\beta}$ and $\sk{D}$ as well as on the fact that weak convergence of random closed sets can be characterized in terms of capacity functionals. Namely, according to \cite[Theorem 1.7.7]{Mol17} a family of random closed sets $(Z_\lambda)_{\lambda>0}$ in $\RR^{d-1}$ converges weakly to a random closed set $Z$ in $\RR^{d-1}$ if and only if
\begin{equation}\label{eq:ConvergenceRACS}
\lim_{\lambda\to\infty}T_{Z_\lambda}(C) = T_{Z}(C)
\end{equation}
for all $C\in\cC$ such that $T_Z(\inter C)=T_{Z}(C)$. The proof of Theorem \ref{tm:CobvergenceGaussian} consists, roughly speaking, of three steps that can be summarized as follows:
\begin{itemize}
\item[] \textbf{Step 1 (Lemma \ref{lm:BetaBoundaryBounds} in Section \ref{sec:TechPrep}):} We recall that the constructions of $\widetilde{\sk{D}}_{\beta}$, $\widetilde{\sk{D}}^{\prime}_{\beta}$ and $\widetilde{\sk{D}}$ on $\RR^{d-1}$ are based on the paraboloid hull processes associated with the Poisson point processes $\zeta_{\beta}$, $\zeta_{\beta}'$ and $\zeta$ on $\RR^{d-1}\times\RR$, respectively. Our first step consists in bounding from above and below the heights of the paraboloids which with high probability determine the tessellation in a ball $B_R\subset\RR^{d-1}$ of some fixed radius $R>0$.
\item[] \textbf{Step 2 (Lemma \ref{lm:StabRadius} in Section \ref{sec:TechPrep}):}
In the second step we consider a kind of localization property in the spatial coordinate. More precisely, for any given $R>0$ we construct some $r>0$ such that the restricted random sets $\widetilde{\sk{D}}_{\beta}\cap B_R$, $\widetilde{\sk{D}}^{\prime}_{\beta}\cap B_R$ and $\widetilde{\sk{D}}\cap B_R$ are determined with high probability once the paraboloid hull processes underlying their construction are known in the cylinder $B_{R+r}\times\RR$.
\item[] \textbf{Step 3 (Section \ref{sec:ProofFinalStep}):} The last step consists of a coupling argument and is essentially based on the estimates we derived in Step 1 and Step 2. Using these results we obtain that with high probability the random sets $\widetilde{\sk{D}}_{\beta}\cap B_R$, $\widetilde{\sk{D}}^{\prime}_{\beta}\cap B_R$ and $\widetilde{\sk{D}}\cap B_R$ are determined by the restrictions of corresponding Poisson point processes to some particular region in $\RR^{d-1}\times\RR$. Thus, for any given $\varepsilon>0$ and $R>0$ we find a sufficiently large parameter $\beta$ and construct a probability space such that with probability $1-\varepsilon$ the random tessellations $\widetilde{\sk{D}}_{\beta}$ and $\widetilde{\sk{D}}^{\prime}_{\beta}$ restricted to $B_R$ \textit{coincide} with the Gaussian limiting tessellation $\widetilde{\sk{D}}$ within the same ball. In particular, this yields the convergence of the capacity functionals $T_\beta$ and $T_\beta'$ of $\widetilde{\sk{D}}_{\beta}$ and $\widetilde{\sk{D}}^{\prime}_{\beta}$, respectively, to the capacity functional $T$ of the Gaussian-Delaunay tessellation.
\end{itemize}
Although the Fell topology on the space of closed subsets of $\RR^{d-1}$ and the related weak convergence are widely used in the theory of random sets (see, for example, the monograph \cite{Mol17}), convergence in Fell topology does not imply convergence of many interesting quantities, such as the number of cells or faces within a given ball, the total surface area within a ball, and so on. However, the proof of Theorem \ref{tm:CobvergenceGaussian} we give and, which we summarized above, actually shows weak convergence of $\widetilde{\sk{D}}_{\beta}$ and $\widetilde{\sk{D}}^{\prime}_{\beta}$ to $\sk{D}$ with respect to a much stronger topology on the space of closed subsets of $\RR^{d-1}$. To introduce this {non-separable} topology, let $C\in \cC$ be a closed subset of $\RR^{d-1}$. The filter of neighbourhoods of this set is defined as follows. Take a radius $R>0$ and consider the set $\cO_R(C)$ consisting of all closed subsets $B\subset\RR^{d-1}$ such that the restrictions of $B$ and $C$ to the ball of radius $R$ around the origin are equal, that is, $B\cap B_R=C\cap B_R$. One easily checks that the system $\cO_R(C)$ in fact defines a filter of neighbourhoods of $C$. A subset of $\cC$ is called open if, together with each element $C\in \cC$ it contains a set of the form $\cO_R(C)$ for some $R>0$. This defines a (very strong) topology on $\cC$ for which the geometric functionals mentioned above are continuous. However, we decided to present Theorem \ref{tm:CobvergenceGaussian} for the Fell topology for simplicity and also since this is the much more classical notion used in the theory of random sets.

\subsection{Technical preparations}\label{sec:TechPrep}

The main idea of the proof of Theorem \ref{tm:CobvergenceGaussian} is to use \eqref{eq:ConvergenceRACS} and the connection between $\widetilde{\sk{D}}_{\beta}$ (and $\widetilde{\sk{D}}^{\prime}_{\beta}$, $\widetilde{\sk{D}}$) and the boundaries of the paraboloid growth and hull processes described in Section \ref{sec:ParabHullProc}. In this section we present a number of technical preparations for the proof of Theorem \ref{tm:CobvergenceGaussian}, which is deferred to the next section.

\begin{lemma}\label{lm:BetaBoundaryBounds}
Fix $A>0$ and let $T,t\in\RR$. The following assertions hold.
	\begin{enumerate}
	\item
	\begin{enumerate}
	\item For any $\beta \ge \beta_0\ge 1$ we have
	$$
	\PP\Big(\sup\limits_{\substack{(v,h)\in\partial \Psi(\zeta_{\beta}),\\v\in B_A}} h > T\Big)<\begin{cases}
	\exp\Big(-{A^{d-1}\over 2^{d/2}\sqrt{\pi}\Gamma({d+1\over 2})}\Big(1+{T-4A^2\over 2\beta_0}\Big)^{\beta_0}\Big) &:T>4A^2\\
	1 &: T\leq 4A^2;
	\end{cases}
	$$
	\item For any $\beta > (d+1)/2$ we have
    $$
	\PP\Big(\sup\limits_{\substack{(v,h)\in\partial \Psi(\zeta^{\prime}_{\beta}),\\v\in B_A}} h > T\Big)\leq
	\begin{cases}
	\exp\Big(-{A^{d-1}\over (2(d+1))^{d/2}\sqrt{\pi}\Gamma({d+1\over 2})}e^{T/2-2A^2}\Big) &: T< 4A^2+2\beta;\\
	0 &: T\ge 4A^2+2\beta;
	\end{cases}
	$$
	\item We have
	$$
	\PP\Big(\sup\limits_{\substack{(v,h)\in\partial \Psi(\zeta),\\v\in B_A}} h > T\Big)<\exp\Big(-{A^{d-1}\over 2^{d/2-1}\sqrt{\pi}\Gamma({d+1\over 2})}e^{T/2-2A^2}\Big).
	$$
	\end{enumerate}
	\item
	\begin{enumerate}
	\item For any $\beta >1 $ we have
	$$
	\PP\Big(\inf\limits_{\substack{(v,h)\in\partial \Psi(\zeta_{\beta}),\\v\in B_A}} h < t\Big)<1-\exp\Big(-{2(d/2+1)^{d/2}\over \sqrt{\pi}}(A+1)^{d-1}e^{t/2}\Big).
	$$
	\item For any $\beta \ge\beta_0> (d+1)/2$ we have 
	$$
	\hspace{-1.5cm}\PP\Big(\inf\limits_{\substack{(v,h)\in\partial \Psi(\zeta_{\beta}^{\prime}),\\v\in B_A}} h < t\Big)\leq\begin{cases}
1-\exp\Big(-{2(2\beta_0)^{\beta_0}(A+1)^{d-1}\over\sqrt{\pi}(2\beta_0-d-1)^{(d+1)/2}}\Big(2\beta_0-t\Big)^{-\beta_0+(d+1)/2}\Big) &: t<0\\
1 &: t\geq 0.
	\end{cases}
	$$
	\item We have
	$$
	\PP\Big(\inf\limits_{\substack{(v,h)\in\partial \Psi(\zeta),\\v\in B_A}} h < t\Big)<1-\exp\Big(-{2\over\sqrt{\pi}}(A+1)^{d-1}e^{t/2}\Big).
	$$
	\end{enumerate}
	\end{enumerate}
\end{lemma}

\begin{figure}[t]
\centering
\begin{tikzpicture}
\begin{axis}[enlargelimits=0.1,axis lines=none]
\addplot[domain=-5:-3.5,thick] {(x+4)^2+3};
\addplot[domain=-3.5:-0.5,thick] {(x+2)^2+1};
\addplot[domain=-0.5:1,thick] {(x-0.5)^2+2.25};
\addplot[domain=1:3.5,thick] {(x-2)^2+1.5};
\addplot[domain=3.5:5,thick] {(x-4)^2+3.5};
\addplot [dashed]coordinates {(-5.5,1.7) (5.5,1.7)};
\end{axis}
\fill (2.39,0.48) circle (0.075);
\draw[dashed] (-1,0) -- (8,0);
\draw[thick,<->] (1,0) -- (5.9,0);
\node[below] at (5.4,0) {$B_A(0)$};
\fill (3.45,0) circle (0.075);
\node[below] at (3.45,0) {$0$};
\draw[dotted] (2.39,0.48) -- (2.39,0);
\node[below] at (2.39,0) {\textcolor{blue}{$v'$}};
\fill (2.39,0.1) circle (0.075);
\draw[thick,<->,blue] (1.95,0.1) -- (2.83,0.1);
\draw[dotted] (1.95,0.1) -- (1.95,1.5);
\draw[dotted] (2.83,0.1) -- (2.83,1.5);
\node[above] at (1,0) {\textcolor{blue}{$B_{\sqrt{T-h'}}(v')$}};
\node[below] at (0.3,1.8) {$T$};
\node[below] at (2.65,0.65) {\textcolor{blue}{$h'$}};
\node[below] at (5.7,3.8) {$\Psi(\xi)$};
\node[below] at (8,0) {$\RR^{d-1}$};
\end{tikzpicture}
\caption{Illustration of the first argument in the proof of Lemma \ref{lm:BetaBoundaryBounds}.}
\label{fig:Proof1}
\end{figure}
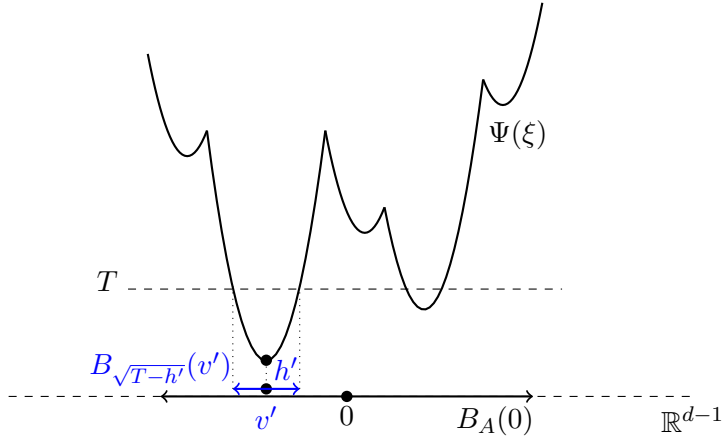

\begin{proof}
In what follows, let us write $\xi$ for one of the point processes $\zeta$, $\zeta_{\beta}$ or $\zeta^{\prime}_{\beta}$. We start by proving the first claim and we observe that by the definition
$$
\partial \Psi(\xi)=\big\{(v,h)\in\RR^{d-1}\times\RR\colon h=\inf\limits_{x\in\xi,(v,t)\in\Pi_{+,x}}t\big\}.
$$
This allows us to write
	\begin{align*}
	\PP\Big(\sup\limits_{\substack{(v,h)\in\partial \Psi(\xi),v\in B_A}} h > T\Big)&\leq \PP\Big(\inf_{x\in\xi}\;\sup\limits_{\substack{(v,h)\in \Pi_{+,x},v\in B_A}} h > T \Big)\\
	&\leq \PP\Big(\inf_{x=(v',h')\in\xi,v'\in B_A}\;\sup\limits_{\substack{(v,h)\in \Pi_{+,x},v\in B_A}} h > T\Big),
	\end{align*}
	where we recall that for a particle $x:=(v',h')\in\xi$ we put $\Pi_{+,x}=\{(v,h)\colon \|v-v'\|^2+h'=h\}$.
	Next, we note that the condition
	$$
	\sup\limits_{\substack{(v,h)\in \Pi_{+,x},v\in B_A}} h > T
	$$
	for $x=(v',h')\in\xi$ with $v'\in B_A$
	can be rephrased by saying that
	$$
	B_A\cap B_{\sqrt{T-h'}}(v')\neq B_A,
	$$
	see Figure \ref{fig:Proof1}. This cannot hold for points $(v',h')$ with $h'\leq T-4A^2$ and $v'\in B_A$, hence, we conclude that
	\begin{equation}\label{eq_01}
	\begin{aligned}
	\PP\Big(\inf_{x=(v',h')\in\xi,v'\in B_A}\;\sup\limits_{\substack{(v,h)\in \Pi_{+,x},v\in B_A}} h > T\Big)&\leq \PP\Big(\xi\cap (B_A\times (-\infty, T-4A^2])=\varnothing\Big)\\
	&=\exp(-\EE[\xi\cap (B_A\times (-\infty, T-4A^2])]),
	\end{aligned}
	\end{equation}
	since $\xi$ is a Poisson point process.
	
	Next, we evaluate the last term for $\xi=\zeta$, $\xi=\zeta_{\beta}$ and $\xi=\zeta^{\prime}_{\beta}$. For $\xi=\zeta$ we have
		\begin{align*}
	 \EE[\zeta&\cap (B_A\times (-\infty, T-4A^2])]= {1\over (2\pi)^{d/2}}\int_{B_A}\int_{-\infty}^{T-4A^2}e^{s/2}\,\dint s\dint w= {2A^{d-1}\over 2^{d/2}\sqrt{\pi}\Gamma({d+1\over 2})}e^{T/2-2A^2},
	\end{align*}
	which together with \eqref{eq_01} finishes the proof in this case.
	Analogously, from the definition of the intensity measure of the Poisson point processes $\zeta_{\beta}$ for $T>4A^2-2\beta$ we get
	\begin{align*}
	 \EE[\zeta_{\beta}\cap (B_A\times (-\infty, T-4A^2])]&= {c_{d,\beta}\over (2\beta)^{d/2}}\,\int_{B_A}\int_{-\infty}^{T-4A^2}{\bf 1}(s\geq -2\beta)\,\Big(1+{s\over 2\beta}\Big)^{\beta}\,\dint s\dint w\\
	 &= {c_{d,\beta}\over (2\beta)^{d/2}}\,{\pi^{d-1\over 2}\over \Gamma({d+1\over 2})}A^{d-1}\int_{-2\beta}^{T-4A^2}\Big(1+{s\over 2\beta}\Big)^{\beta}\,\dint s\\
	 &= {c_{d,\beta}\over (2\beta)^{d/2}}\,{2\beta\over \beta +1}{\pi^{d-1\over 2}\over \Gamma({d+1\over 2})}A^{d-1}\Big(1+{T-4A^2\over 2\beta}\Big)^{\beta+1}.
	\end{align*}
	Applying inequality in \cite[Paragraph 5.6.8]{NIST} for quotients of gamma functions we conclude that
	\begin{equation}\label{eq_08}
	{c_{d,\beta}\over \beta^{d/2}}={\Gamma\left({d\over 2}+\beta+1\right)\over (\pi\beta)^{d/2}\Gamma(\beta+1)} \geq {\beta^{d/2}\over (\pi\beta)^{d/2}} = {1\over \pi^{d/2}}.
	\end{equation}
	Moreover since the function $x\mapsto(1+x)^{1/x}$ is decreasing for $x>-1$ we get for $\beta \ge \beta_0\ge 1$ and $T\ge 4A^2$ that
	\begin{align*}
	 \Big(1+{T-4A^2\over 2\beta}\Big)^{\beta+1}&=\Big(1+{T-4A^2\over 2\beta}\Big)^{{2\beta\over T-4A^2}\cdot {T-4A^2\over 2}}\Big(1+{T-4A^2\over 2\beta}\Big)\ge \Big(1+{T-4A^2\over 2\beta_0}\Big)^{\beta_0}.
	\end{align*}	
	Finally, by definition of the intensity measure of the Poisson point processes $\zeta_{\beta}^{\prime}$ and the fact that $1-x < e^{-x}$ for $x<1$ we obtain for $T<4A^2+2\beta$,
	\begin{align*}
	 \EE[\zeta^{\prime}_{\beta}\cap (B_A\times (-\infty, T-4A^2])]&= {c^{\prime}_{d,\beta}\over (2\beta)^{d/2}}\,\int_{B_A}\int_{-\infty}^{T-4A^2}\Big(1-{s\over 2\beta}\Big)^{-\beta}\,\dint s\dint w\\
	 &> {c^{\prime}_{d,\beta}\over (2\beta)^{d/2}}\,\int_{B_A}\int_{-\infty}^{T-4A^2}e^{s/2}\,\dint s\dint w\\
	 &={c^{\prime}_{d,\beta}\over (2\beta)^{d/2}}{\pi^{d-1\over 2}\over \Gamma({d+1\over 2})}A^{d-1}e^{T/2-2A^2}.
	\end{align*}
	As before, applying the inequality in \cite[Paragraph 5.6.8]{NIST} for quotients of gamma functions we get
$$
	 \EE[\zeta^{\prime}_{\beta}\cap (B_A\times (-\infty, T-4A^2])]>{A^{d-1}\over (2(d+1))^{d/2}\sqrt{\pi}\Gamma({d+1\over 2})}e^{T/2-2A^2}.
$$
For $T\ge 4A^2+2\beta$ we have $\EE[\zeta^{\prime}_{\beta}\cap (B_A\times (-\infty, T-4A^2])]=\infty$. Substituting these estimates into \eqref{eq_01} we finish the proof of the first claim.

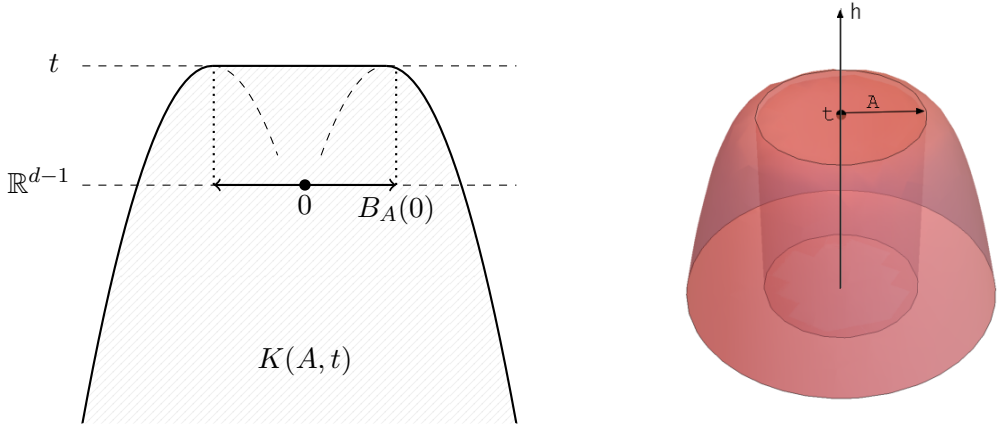
\begin{figure}[t]
	\centering
	\begin{tikzpicture}
	\begin{axis}[enlargelimits=0.1,axis lines=none]
	\addplot[domain=-10:-4,name path=f1,thick] {-0.25*(x+4)^2+3};
	\addplot[domain=-4:-1,dashed,,name path=f2] {-0.25*(x+4)^2+3};
	\addplot[domain=4:10,name path=g1,thick] {-0.25*(x-4)^2+3};
	\addplot[domain=1:4,dashed,,name path=g2] {-0.25*(x-4)^2+3};
	\addplot[domain=-10:-4,name path=h1,white] {-6};
	\addplot[domain=-4:4,name path=h2,white] {-6};
	\addplot[domain=4:10,name path=h5,white] {-6};
	\addplot[color=black,name path=k,thick] coordinates {(-4,3) (4,3)};
	
	\tikzfillbetween[of=f1 and h1]{gray, opacity=0.25,pattern=north east lines};
	\tikzfillbetween[of=g1 and h5]{gray, opacity=0.25,pattern=north east lines};
	\tikzfillbetween[of=k and h2]{gray, opacity=0.25,pattern=north east lines};
	
	\addplot[color=black,dashed] coordinates {(-10,3) (10,3)};
	\addplot[color=black,dashed] coordinates {(-10,0) (10,0)};
	
	\end{axis}
	\node[above] at (0.2,5) {$t$};
	\node[above] at (0,3.4) {$\RR^{d-1}$};
	\node[below] at (3.5,3.64) {$0$};
	\fill (3.5,3.64) circle (0.075);
	\draw[thick,<->] (2.3,3.64) -- (4.7,3.64);
	\draw[dotted,thick] (2.3,3.64) -- (2.3,5.2);
	\draw[dotted,thick] (4.7,3.64) -- (4.7,5.2);
	\node[below] at (3.5,1.64) {$K(A,t)$};
	\node[below] at (4.7,3.64) {$B_A(0)$};
	\end{tikzpicture}
	\includegraphics[width=0.4\columnwidth]{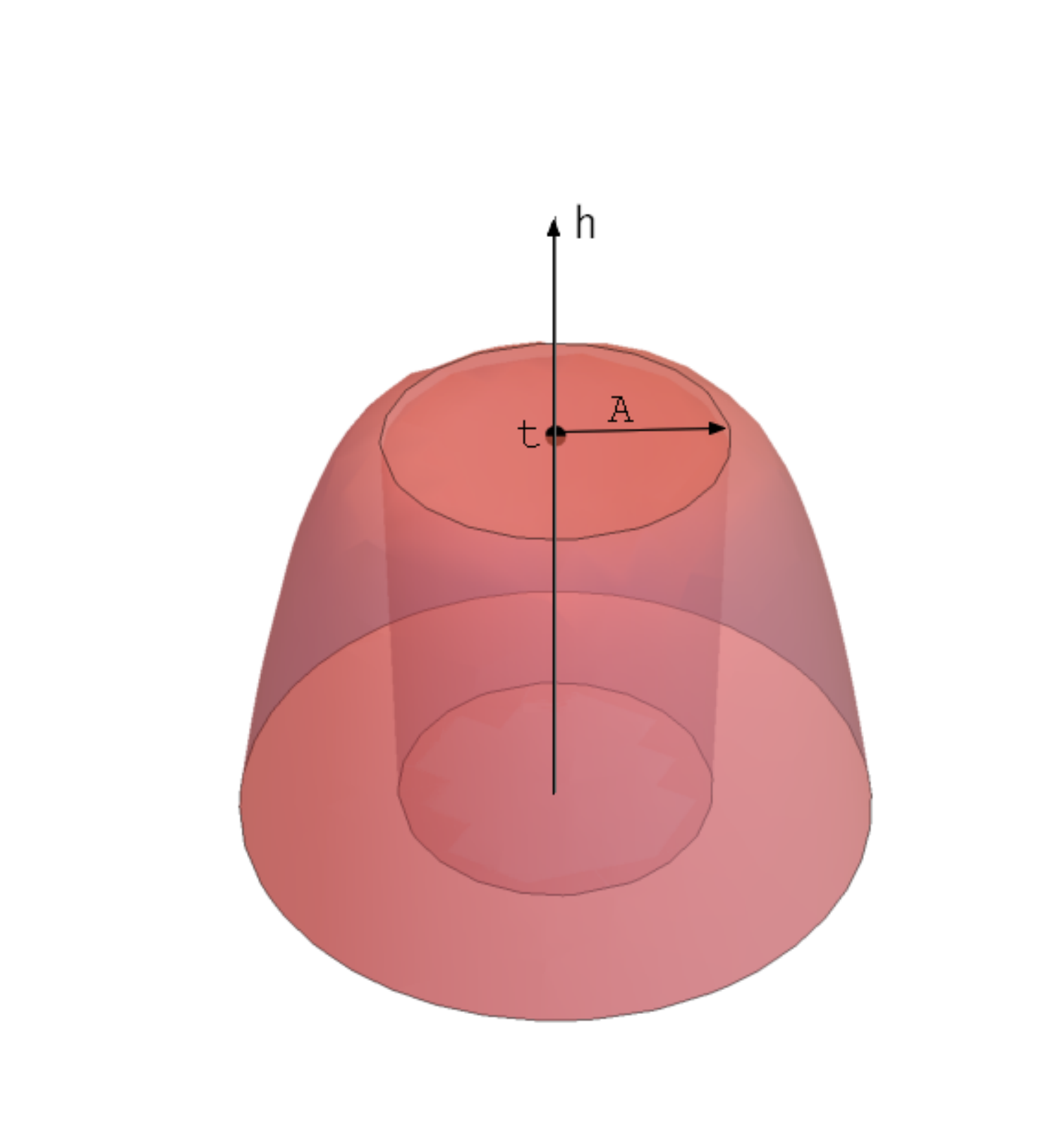}
	\caption{The region $K(A,t)$ in $\RR^2$ (left) and $\RR^3$ (right).}
	\label{fig:Proof2}
\end{figure}

In order to prove the second claim we argue similarly and start by writing
\begin{equation}\label{eq_03}
	\begin{aligned}
	\PP\Big(\inf\limits_{\substack{(v,h)\in\partial \Psi(\xi),v\in B_A}} h < t\Big)&= \PP\Big(\inf_{x\in\xi}\;\inf\limits_{\substack{(v,h)\in \Pi_{+,x},v\in B_A}} h <t \Big)\\
	&=1- \PP\Big(\inf_{x\in\xi}\;\inf\limits_{\substack{(v,h)\in \Pi_{+,x},v\in B_A}} h > t\Big)\\
	&=1- \PP\Big(\inf\limits_{\substack{(v,h)\in \Pi_{+,x},v\in B_A}} h > t \text{ for all } x\in \xi\Big).
	\end{aligned}
	\end{equation}
	As before, given a particle $x:=(v',h')\in\xi$ the condition that $\inf\limits_{\substack{(v,h)\in \Pi_{+,x},v\in B_A}} h > t$ for $x=(v',h')\in\xi$
	is equivalent to $h'>t$  if $ v'\in B_{A}$ and to $h'+\Big\|v'-A{v'\over \|v'\|}\Big\|^2 > t$ if $v'\not\in B_{A}$.
Consider the region
$$
K(A,t):=(B_A\times (-\infty,t])\cup \Big(\bigcup\limits_{w\in\partial B_A}\Pi_{-,(w,t)}^{\downarrow}\Big),
$$
which in $\RR^2$ and $\RR^3$ is illustrated in Figure \ref{fig:Proof2}. Then
\begin{equation}\label{eq_02}
\PP\Big(\inf\limits_{\substack{(v,h)\in \Pi_{+,x},v\in B_A}} h > t \text{ for all } x\in \xi\Big)= \PP\big(\xi\cap K(A,t)=\varnothing\big)=\exp(-\EE[\xi\cap K(A,t)])
\end{equation}
and it remains to determine $\EE[\xi\cap R(A,t)]$ for $\xi=\zeta$, $\xi=\zeta_{\beta}$ and $\xi^{\prime}=\zeta_{\beta}$. We start by analysing the second case, where Fubini's theorem implies that
	\begin{align*}
	 \EE[\zeta_{\beta}\cap K(A,t)]&= {c_{d,\beta}\over (2\beta)^{d/2}}\int_{-\infty}^{t}\int_{\{w\in\RR^{d-1}:\|w\|\leq \sqrt{t-s}+A\}}{\bf 1}(s\geq -2\beta)\Big(1+{s\over 2\beta}\Big)^{\beta}\,\dint w\dint s\\
	 &={c_{d,\beta}\pi^{d-1\over 2}\over (2\beta)^{d/2}\Gamma({d+1\over 2})}\int_{-\infty}^{t}{\bf 1}(s\geq -2\beta)(\sqrt{t-s}+A)^{d-1}\Big(1+{s\over 2\beta}\Big)^{\beta}\,\dint s.
	 \end{align*}
	 This time, we use Gautschi's inequality \cite[Paragraph 5.6.4]{NIST} and the recurrence formula for the gamma function from \cite[Paragraph 5.5.1]{NIST} to conclude that for $\beta > 1$,
	\begin{equation}\label{eq_8}
	 {c_{d,\beta}\over \beta^{d/2}}={(d/2+\beta)(d/2+\beta -1)\cdots (\beta +1+ \{d/2\})\Gamma\left(\beta+1+\{d/2\}\right)\over \pi^{d/2}\beta^{d/2}\Gamma(\beta+1)}<{(d/2+1)^{d/2}\over \pi^{d/2}},
	\end{equation}
	where $\{x\}$ stands for the fractional part of a real number $x\in\RR$.
	And since $1+x\leq e^x$ for $x>-1$ we see that
		\begin{align*}
	 \EE[\zeta_{\beta}\cap K(A,t)]&\leq  {(d/2+1)^{d/2}\over 2^{d/2}\sqrt{\pi}\Gamma({d+1\over 2})}\int_{-\infty}^{t}(\sqrt{t-s}+A)^{d-1}e^{s/2}\,\dint s\\
	 &=   {2(d/2+1)^{d/2}\over 2^{d/2}\sqrt{\pi}\Gamma({d+1\over 2})}e^{t/2}\int_{0}^{\infty}(\sqrt{2y}+A)^{d-1}e^{-y}\,\dint y\\
	 &=  {2(d/2+1)^{d/2}\over \sqrt{\pi}\Gamma({d+1\over 2})}e^{t/2}\sum\limits_{i=0}^{d-1}{d-1\choose i}2^{(i-d)/2}A^{d-1-i}\int_{0}^{\infty}y^{i/2}e^{-y}\,\dint y\\
	 &<  {2(d/2+1)^{d/2}\over \sqrt{\pi}}e^{t/2}\sum\limits_{i=0}^{d-1}{d-1\choose i}A^{d-1-i}{\Gamma\Big({i\over 2}+1\Big)\over\sqrt{2} \Gamma({d+1\over 2})}\\
	 &<  {2(d/2+1)^{d/2}\over \sqrt{\pi}}(A+1)^{d-1}e^{t/2}.
	 \end{align*}
	
	 Applying the same arguments to the Poisson point process $\zeta^{\prime}_{\beta}$  we obtain for $t<2\beta$ that
	\begin{equation}\label{eq_13}
	\begin{aligned}
	 \EE[\zeta^{\prime}_{\beta}\cap K(A,t)]&={c^{\prime}_{d,\beta}\pi^{d-1\over 2}\over (2\beta)^{d/2}\Gamma({d+1\over 2})}\int_{-\infty}^{t}(\sqrt{t-s}+A)^{d-1}\Big(1-{s\over 2\beta}\Big)^{-\beta}\,\dint s\\
	 &={c^{\prime}_{d,\beta}\pi^{d-1\over 2}\over (2\beta)^{d/2}\Gamma({d+1\over 2})}\sum\limits_{i=0}^{d-1}{d-1\choose i}A^{d-i}\int_{-\infty}^{t}(t-s)^{i/2}\Big(1-{s\over 2\beta}\Big)^{-\beta}\,\dint s.
	 \end{aligned}
	 \end{equation}
	 Next, for $i\in\{0,1,\ldots,d-1\}$ we compute the integral appearing in the last expression:
	 \begin{align*}
	 \int_{-\infty}^{t}(t-s)^{i/2}\Big(1-{s\over 2\beta}\Big)^{-\beta}\,\dint s&=2^{i/2+1}\int_{0}^{\infty}y^{i/2}\Big(1-{t\over 2\beta}+{y\over \beta}\Big)^{-\beta}\,\dint y\\
	&=2^{i/2+1}\beta^{\beta}\Big(\beta-{t\over 2}\Big)^{-\beta+1+i/2}\int_{0}^{\infty}z^{i/2}(1-z)^{-\beta}\,\dint z\\
	&=2^{i/2+1}\beta^{\beta}\Big(\beta-{t\over 2}\Big)^{-\beta+1+i/2}\int_{0}^{1}(1-x)^{i/2}x^{\beta-i/2}\,\dint x\\
	&=2^{i/2+1}\Gamma(i/2+1){\beta^{\beta}\Gamma(\beta-1-i/2)\over \Gamma(\beta)}\Big(\beta-{t\over 2}\Big)^{-\beta+1+i/2}.
\end{align*}
Again, by inequality \cite[Paragraph 5.6.8]{NIST} for $\beta\ge\beta_0>(d+1)/2$ we obtain
$$
{\beta^{i/2+1}\Gamma(\beta-1-i/2)\over \Gamma(\beta)}<\Big(1-{d+1\over 2\beta_0}\Big)^{-(d+1)/2},
$$
and, thus,
	 \begin{align*}
	 \int_{-\infty}^{t}(t-s)^{i/2}\Big(1-{s\over 2\beta}\Big)^{-\beta}\,\dint s&<2^{i/2+1}\Big(1-{d+1\over 2\beta_0}\Big)^{-(d+1)/2}\Gamma(i/2+1)\Big(1-{t\over 2\beta}\Big)^{-\beta+1+{i\over 2}}.
\end{align*}
Moreover since $(1+x)^{1/x}$ is decreasing for $x>-1$ and $i\leq d-1$ we have for $t<0$ and $\beta \ge\beta_0> (d+1)/2$ that
\begin{align*}
\Big(1-{t\over 2\beta}\Big)^{-\beta+1+i/2}&=\Big(1-{t\over 2\beta}\Big)^{-{2\beta\over t}{t\over 2}}\Big(1-{t\over 2\beta}\Big)^{1+i/2}<\Big(1-{t\over 2\beta_0}\Big)^{-\beta_0+(d+1)/2}.
\end{align*}
Combining this with $c^{\prime}_{d,\beta}\beta^{-d/2}<\pi^{-d/2}$ and \eqref{eq_13} we obtain
\begin{align*}
\EE[\zeta^{\prime}_{\beta}\cap K(A,t)]&<{2\over\sqrt{\pi}}\Big(1-{d+1\over 2\beta_0}\Big)^{-(d+1)/2}\Big(1-{t\over 2\beta_0}\Big)^{-\beta_0+(d+1)/2}\\
&\qquad\qquad\qquad\times\sum\limits_{i=0}^{d-1}{d-1\choose i}2^{(i-d+1)/2}A^{d-1-i}{\Gamma({i\over 2}+1)\over \sqrt{2}\Gamma({d+1\over 2})}\\
&<  {2\over\sqrt{\pi}}\Big(1-{d+1\over 2\beta_0}\Big)^{-(d+1)/2}(A+1)^{d-1}\Big(1-{t\over 2\beta_0}\Big)^{-\beta_0+(d+1)/2}\\
&= {2(2\beta_0)^{\beta_0}\over\sqrt{\pi}(2\beta_0-d-1)^{(d+1)/2}}(A+1)^{d-1}\Big(2\beta_0-t\Big)^{-\beta_0+(d+1)/2}.
\end{align*}

	 With a similar computation one also shows that
	 \begin{align*}
	 \EE[\zeta\cap K(A,t)]={1\over 2^{d/2}\sqrt{\pi}\Gamma({d+1\over 2})}\int_{-\infty}^{t}(\sqrt{t-s}+A)^{d-1}e^{s/2}\,\dint s<{2\over\sqrt{\pi}}(A+1)^{d-1}e^{t/2}.
	 \end{align*}
	This completes the proof of the lemma.
\end{proof}

In the next lemma we prove a kind of localization property for $\beta$-, $\beta^{\prime}$- and Gaussian-Delaunay tessellations. It is very much in the spirit of the geometric limit theory of stabilization for which we refer to the survey articles \cite{SchreiberSurvey,YukichSurvey}. For $R\geq 1$ and $r>0$ we say that $\widetilde{\sk{D}}_{\beta}\cap B_R$ \textit{is determined by particles $(v,h)\in\partial \Psi(\zeta_{\beta})$ with $v\in B_{R+r}$}, provided that the $\beta$-Delaunay tessellation within $B_R$ is unaffected by changes of the point configuration $\zeta_{\beta}$ outside of $B_{R+r}\times \RR$. The same terminology is also applied if $\zeta_{\beta}$ is replaced by one of the point processes $\zeta^{\prime}_\beta$ or $\zeta$.

\begin{lemma}\label{lm:StabRadius}
\begin{enumerate}
\item For any $\varepsilon \in (0,1)$, $R\ge 1$ and $\beta\ge 1$ there exists $r_0:=r_0(\varepsilon, R)>0$ such that for any $r>r_0$ we have
$$
\PP(\widetilde{\sk{D}}_{\beta}\cap B_R \text{ is determined by particles } (v,h)\in\partial \Psi(\zeta_{\beta})\text{ with } v\in B_{R+r})> 1- \varepsilon;
$$
\item for any $\varepsilon \in (0,1)$, $R\ge 1$ and $\beta\ge 3(d+1)/2$ there exists $r^{\prime}_0:=r^{\prime}_0(\varepsilon, R)>0$ such that for any $r>r^{\prime}_0$ we have
$$
\PP(\widetilde{\sk{D}}^{\prime}_{\beta}\cap B_R \text{ is determined by particles } (v,h)\in\partial \Psi(\zeta_{\beta}^{\prime})\text{ with } v\in B_{R+r})> 1- \varepsilon;
$$
\item for any $\varepsilon \in (0,1)$ and $R\ge 1$ there exists $\tilde r_0:=\tilde r_0(\varepsilon, R)>0$ such that for any $r>\tilde r_0$ we have
$$
\PP(\sk{D}\cap B_R \text{ is determined by particles } (v,h)\in\partial \Psi(\zeta)\text{ with }v\in B_{R+r})> 1- \varepsilon.
$$
\end{enumerate}
\end{lemma}

\begin{proof}
In what follows, let us write $\xi$ for one of the Poisson point processes $\zeta$, $\zeta_{\beta}$ or $\zeta^{\prime}_{\beta}$ and let $\widetilde{\sk{D}}(\zeta):=\sk{D}$,  $\widetilde{\sk{D}}(\zeta_{\beta}):=\widetilde{\sk{D}}_{\beta}$, $\widetilde{\sk{D}}(\zeta^{\prime}_{\beta}):=\widetilde{\sk{D}}_{\beta}^{\prime}$. For $R\geq 1$ and $r>0$ consider the event
$$
E(\xi):=\{\widetilde{\sk{D}}(\xi)\cap B_R \text{ is not determined by particles } (v,h)\in\partial \Psi(\xi)\text{ with }v\in B_{R+r}\}.
$$
By Lemma \ref{lm:BetaBoundaryBounds} and the law of total probability we have that
\begin{equation}\label{eq_06}
\begin{aligned}
\PP(E(\xi))&= \PP\big(E(\xi) \, \big| \, \inf\limits_{\substack{(v,h)\in\partial \Psi(\xi),v\in B_{R+r}}} h > t\big)\PP\big(\inf\limits_{\substack{(v,h)\in\partial \Psi(\xi),v\in B_{R+r}}} h > t\big)\\\
&\qquad\qquad+\PP\big(E(\xi) \, \big| \, \inf\limits_{\substack{(v,h)\in\partial \Psi(\xi),v\in B_{R+r}}} h < t\big)\PP\big(\inf\limits_{\substack{(v,h)\in\partial \Psi(\xi),v\in B_{R+r}}} h < t\big)\\
&\leq P_1(\xi)+P_2(\xi),
\end{aligned}
\end{equation}
where
\begin{align*}
P_1(\xi)&:= \PP\big(E(\xi) \, \big| \, \inf\limits_{\substack{(v,h)\in\partial \Psi(\xi),v\in B_{R+r}}} h > t\big),\\
P_2(\xi)&:=\PP\big(\inf\limits_{\substack{(v,h)\in\partial \Psi(\xi),v\in B_{R+r}}} h < t\big).
\end{align*}

According to the construction, $\widetilde{\sk{D}}(\xi)$ coincides almost surely with the skeleton of the random tessellation $\cM_{\Phi}(\xi)$, recall \eqref{eq:DefTPsi}. Then $\widetilde{\sk{D}}(\xi)\cap B_R$ is determined as soon as we know the location of all vertices of paraboloid facets of $\Phi(\xi)$ hitting the set $B_R\times \RR$ and the event $E(\xi)$ occurs if and only if there is a paraboloid facet of $\Phi(\xi)$ hitting the set $B_R\times \RR$ and having a vertex $(v,h)$ with $v\not \in B_{R+r}$. Let $\Pi_{-}(v',h')$ be a paraboloid such that $F(v',h'):=\Pi_{-}(v',h')\cap \Phi(\xi)$ is a paraboloid facet of $\Phi(\xi)$. Assume that $(B_R\times \RR)\cap F(v',h')\neq \varnothing$ and there is a vertex $(v,h)\in F(v',h')$ with $v\not \in B_{R+r}$. Further we note that the set $F(v',h')$ is connected, paraboloid convex, $F(v',h')\subset \partial \Phi(\xi)$ and
$$
\inf\limits_{\substack{(v,h)\in\partial \Phi(\xi),v\in B_{R+r}}} h\ge \inf\limits_{\substack{(v,h)\in\partial \Psi(\xi),v\in B_{R+r}}} h.
$$
\begin{figure}[t]
	\centering
	\includegraphics[width=0.6\columnwidth]{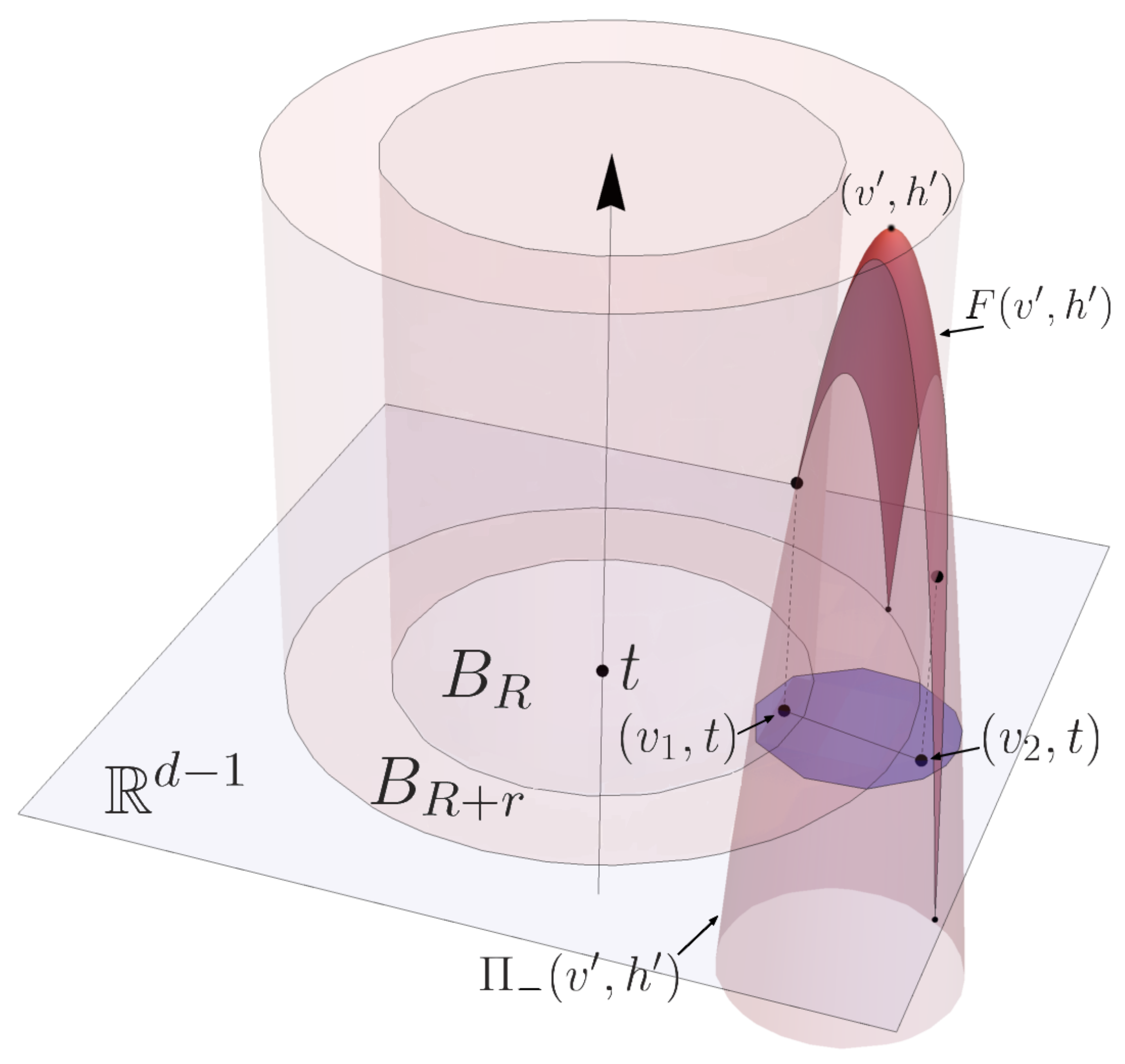}
	\caption{Illustration of the paraboloid facet $F(v',h')$ hitting the set $B_R\times\RR$ and having a vertex $(v,h)$ with $v\not\in B_{R+r}$.}
	\label{fig:proof_fig1}
\end{figure}

This implies that conditioned on
$$
\inf\limits_{\substack{(v,h)\in\partial \Psi(\xi),v\in B_{R+r}}} h > t
$$
we have
$$
F(v',h')\cap (B_{R+r}\times \RR)=F(v',h')\cap (B_{R+r}\times (t,\infty))\neq \varnothing,
$$
and, hence, there are two points $v_1\in \RR^{d-1}\setminus B_{R+r}$ and $v_2\in B_R$ such that for any $v_s=sv_1+(1-s)v_2$, $s\in[0,1]$ we have $(\{v_s\}\times [t,\infty))\cap F(v',h')\neq \varnothing$. Then for any $s\in[0,1]$ we have $v_s\in \Pi_{-}(v',h')\cap (\RR^{d-1}\times \{t\})$, which in turn means that $B_{\sqrt{h'-t}}(v')\cap B_R\neq \varnothing$ and $B_{\sqrt{h'-t}}(v')\cap (\RR^{d-1}\setminus B_{R+r})\neq \varnothing$, see Figure \ref{fig:proof_fig1}.

In this case we obtain the following estimate:
\begin{align*}
P_1(\xi)&\leq \PP\big(\exists\,\text{ a paraboloid facet } F(v',h') { of } \Phi(\xi)\colon \\
&\hspace{1cm}B_{\sqrt{h'-t}}(v')\cap B_R\neq \varnothing, B_{\sqrt{h'-t}}(v')\cap (\RR^{d-1}\setminus B_{R+r})\neq \varnothing\big).
\end{align*}
Next, we observe that the condition $B_{\sqrt{h'-t}}(v')\cap B_R\neq \varnothing$ is equivalent to
$$
(v',h')\in \{(v,h)\in \RR^{d-1}\times [t,\infty)\colon v\in B_{R+\sqrt{h-t}}\}=:K_1(R,t),
$$
and the condition $B_{\sqrt{h'-t}}(v')\cap (\RR^{d-1}\setminus B_{R+r})\neq \varnothing$ is equivalent to
$$
(v',h')\in \{(v,h)\in \RR^{d-1}\times [t,\infty)\colon v\in (\RR^{d-1}\setminus B_{R+r-\sqrt{h-t}})\}=:K_2(R+r,t),
$$
see Figure \ref{fig:proof_fig2}.
Then
\begin{align}
P_1(\xi)&\leq \PP(\exists\,\text{ a paraboloid facet } F(v',h') \text{ of } \Phi(\xi)\colon (v',h')\in K_1(R,t)\cap K_2(R+r,t)),\label{eq_5}
\end{align}

For some $a>0$ consider a partition of $\RR^{d-1}$ into the boxes $Q_{x,y}$, $x\in a\ZZ^{d-1}$, $y\in\ZZ$ of the form
$$
Q_{x,y}:=(x\oplus [0,a]^{d-1})\times [t+r^2/4+y,t+r^2/4+y+1],
$$
where $\oplus$ denotes Minkowski addition. For fixed $y\in \ZZ$, $y\ge 0$, denote by $Q_{x_1,y},\ldots, Q_{x_{m(y)},y}$ the $m(y)$ boxes having non-empty intersection with the set
\begin{equation}\label{eq_13.11.20_eq2}
\begin{aligned}
K(t,R,r,y)&:=\{(v,h)\in K_1(R,t)\cap K_2(R+r,t)\colon  h \in [t+r^2/4+y,t+r^2/4+y+1]\}\\
&\subset\{(v,h)\colon  h \in [t+r^2/4+y,t+r^2/4+y+1], v\in B_{R+\sqrt{h-t}}\}.
\end{aligned}
\end{equation}
Note that for $y <0$ the set $K(t,R,r,y)$ is empty, that is why we can restrict ourself to $y\geq 0$.

\begin{figure}[t]
	\centering
	\includegraphics[width=0.3\columnwidth]{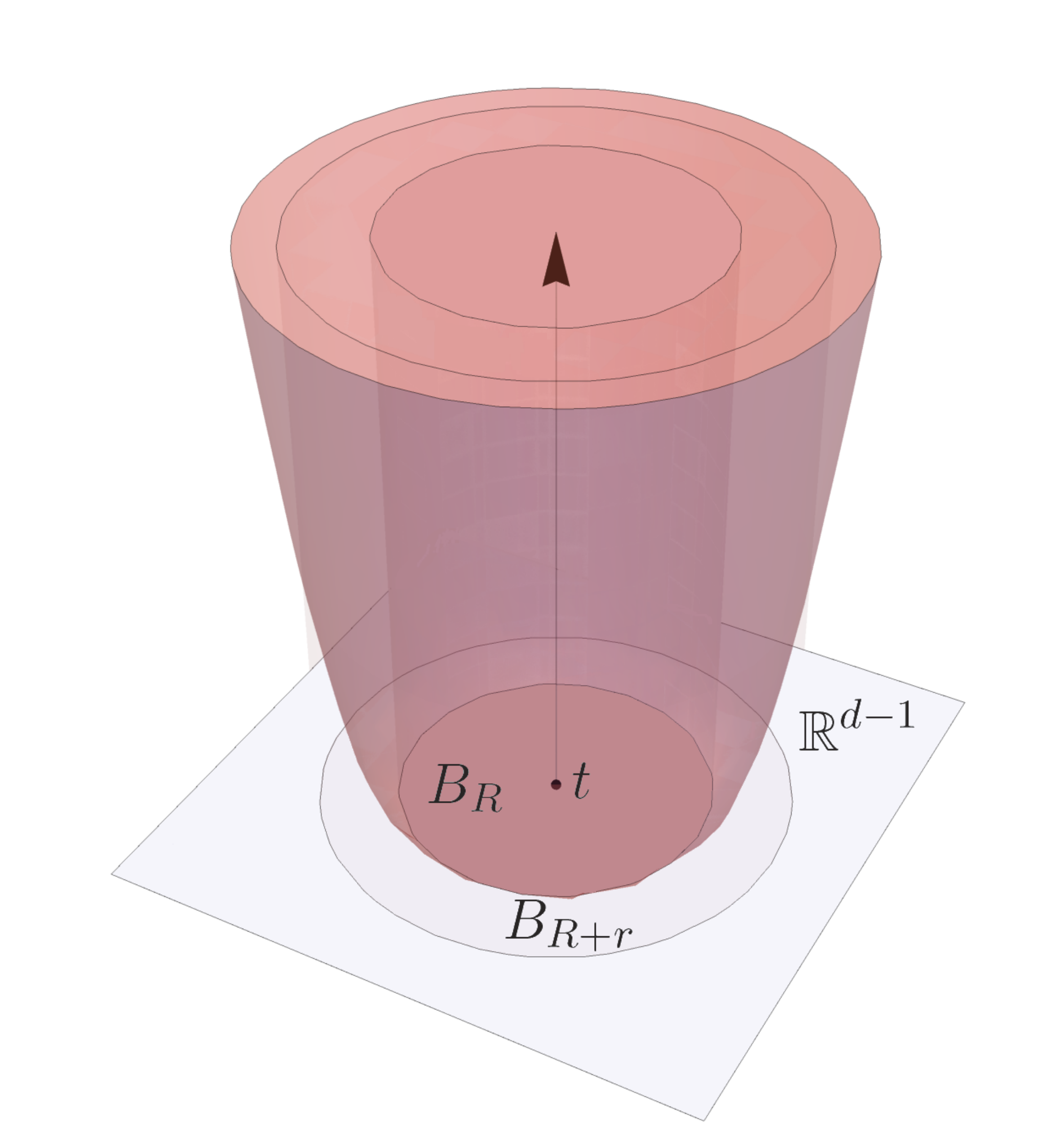}\includegraphics[width=0.3\columnwidth]{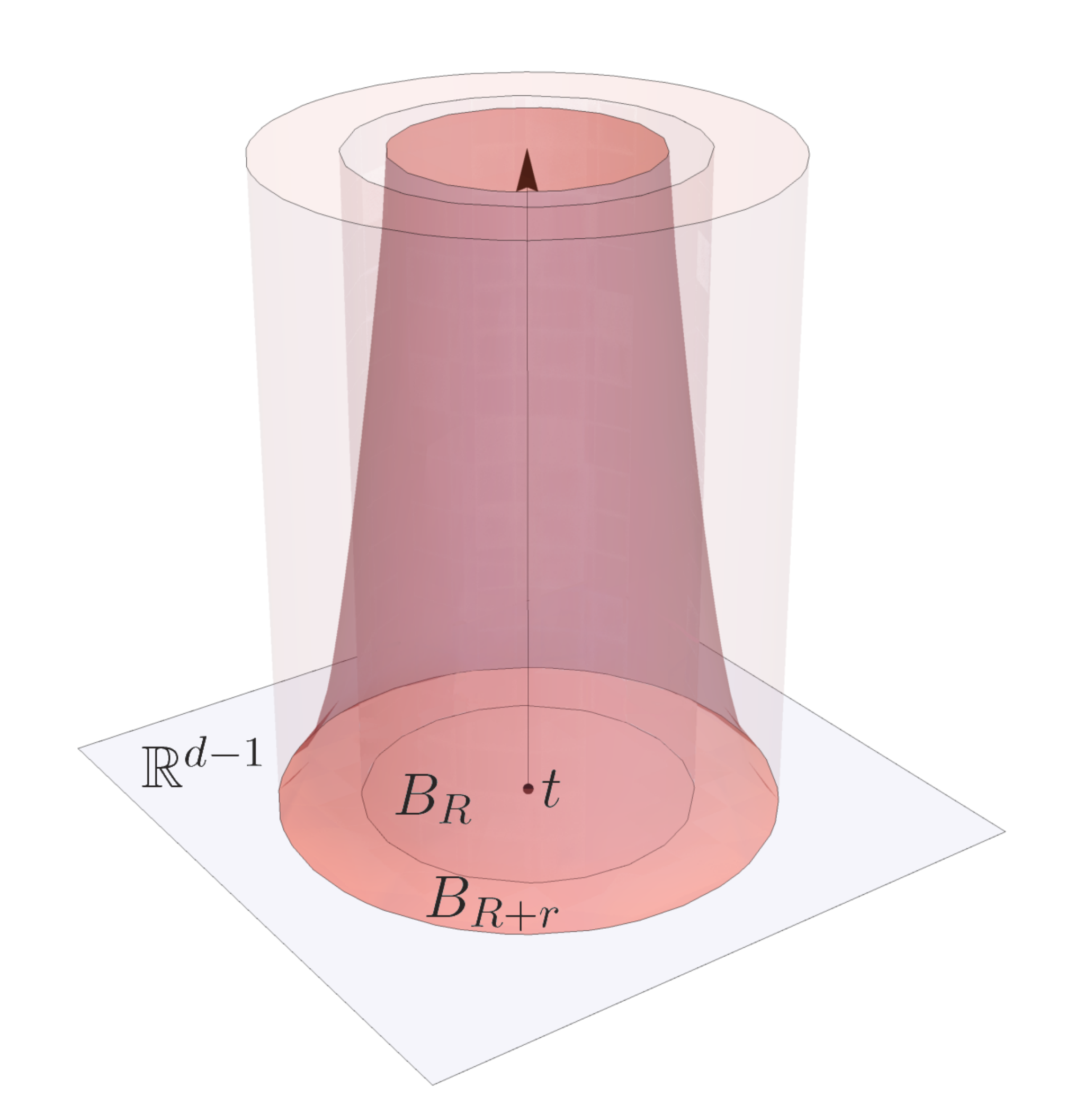}\includegraphics[width=0.3\columnwidth]{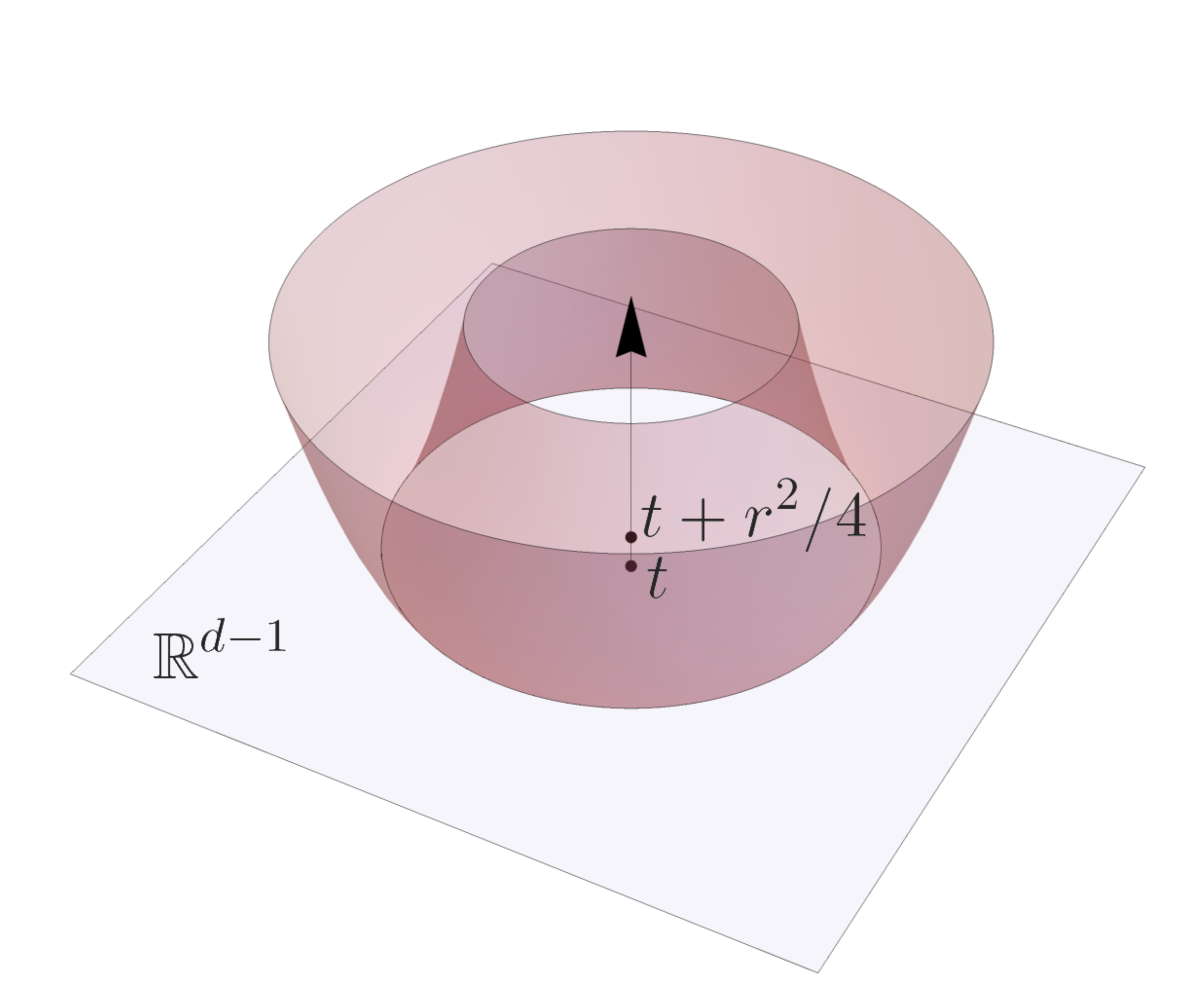}
	\caption{Illustration of the set $K_1(R,t)$ (left),  the complement of the set $K_2(R+r,t)$ (centre) and the set $K_1(R,t)\cap K_2(R+r,t)$ (right).}
	\label{fig:proof_fig2}
\end{figure}

For $y\in\ZZ$, $y\geq 0$ and any $x\in \RR^{d-1}$ we have
\begin{align*}
\PP\Big(\sup\limits_{\substack{(v,h)\in\partial \Psi(\xi),v\in (x\oplus [0,a]^{d-1})}} h > t+r^2/4+y\Big)&=\PP\Big(\sup\limits_{\substack{(v,h)\in\partial \Psi(\xi),v\in [0,a]^{d-1}}} h > t+r^2/4+y\Big)\\
&\leq \PP\Big(\sup\limits_{\substack{(v,h)\in\partial \Psi(\xi),v\in B_{\sqrt{d-1}a}}} h > t+r^2/4+y\Big)\\
&=:p_a(\xi,y).
\end{align*}
Further, applying the Boole's inequality to \eqref{eq_5} and using the fact that apexes of the parabolic facets of paraboloid hull process $\Phi(\xi)$ belong to the boundary of paraboloid growth process $\Psi(\xi)$ (see also Figure \ref{fig:GrowthHullProcess}) we see that
\begin{equation}\label{eq_13.11.20_eq1}
\begin{aligned}
P_1(\xi)&< \sum_{y=0}^{\infty}\sum_{j=0}^{m(y)}\PP(\exists\,\text{ a paraboloid facet } F(v',h') \text{ of } \Phi(\xi)\colon (v',h')\in Q_{x_j,y})\\
&< \sum_{y=0}^{\infty}\sum_{j=0}^{m(y)}\PP\Big(\sup\limits_{\substack{(v,h)\in\partial \Psi(\xi),v\in (x_j\oplus [0,a]^{d-1})}} h > t+r^2/4+y\Big)\\
&<\sum_{y=0}^{\infty}m(y)p_{a}(\xi,y).
\end{aligned}
\end{equation}
It remains to estimate $m(y)$. Since all boxes intersecting $K(t,R,r,y)$ for fixed $y$ are included in the extended set
\begin{align*}
\bigcup\limits_{1\leq i\leq m(y)}Q_{x_i,y}\subset \tilde K(t,R,r,y)&:=[t+r^2/4+y,t+r^2/4+y+1]\times B_{R+\sqrt{r^2/4+y+1}+\sqrt{d-1}a},
\end{align*}
we can conclude, that,
$$
m(y)\leq a^{-d+1}\Vol(\tilde K(t,R,r,y))=\kappa_{d-1}a^{-d+1}\big(R+\sqrt{r^2/4+y+1}+\sqrt{d-1}a\big)^{d-1}.
$$
Substituting this estimate into \eqref{eq_13.11.20_eq1} we obtain
$$
P_1(\xi)\leq\sum_{y=0}^{\infty}\sum_{j=1}^{m(y)}p(\xi,y)<\kappa_{d-1}a^{-d+1}\sum_{y=0}^{\infty}\big(R+\sqrt{r^2/4+y+1}+\sqrt{d-1}a\big)^{d-1}p_a(\xi,y).
$$

First consider the case $\xi=\zeta_{\beta}$. Using estimate 1(a) of Lemma \ref{lm:BetaBoundaryBounds} for $r^2>-4t+16(d-1)a^2$, $\beta_0=1$ and $(\sqrt{d-1}a)^{d-1}=2^{d/2}\sqrt{\pi}\Gamma({d+1\over 2})$ we obtain
\begin{align*}
p(\zeta_\beta,y)\ll_de^{-t/2-r^2/8-y/2}
\end{align*}
and for $R\ge 1$ we have
\begin{align*}
P_1(\zeta_\beta)&\ll_d \sum_{y=0}^{\infty} (R+\sqrt{r^2/4+y+1}+\sqrt{d-1}a)^{d-1}e^{-t/2-r^2/8-y/2}\\
&\ll_d\sum_{i=0}^{d-1}{d-1 \choose i}R^{d-1-i}\sum_{y=0}^{\infty} (r^2/4+y+1)^{i\over 2}e^{-t/2-r^2/8-y/2}.
\end{align*}
For $0\leq i\leq d-1$ consider the functions
$$
S_i(y):=(r^2/4+y+1)^{i\over 2}e^{-t/2-r^2/8-y/2}.
$$
Let us point out that for $r>\sqrt{2d}\ge 2$ all functions $y\mapsto S_i(y)$, $0\leq i\leq d-1$ are strictly decreasing in $y\ge 0$ and the sum can be estimated as
\begin{equation}\label{eq:10.05.20_1}
\begin{aligned}
P_1(\zeta_\beta)&\ll_d\sum_{i=0}^{d-1}{d-1 \choose i}R^{d-1-i}\Big(\int_{0}^{\infty}(r^2/4+y+1)^{i\over 2}e^{-t/2-r^2/8-y/2}\dd y + (r^2/4+1)^{i\over 2}e^{-t/2-r^2/8}\Big)\\
&\ll_de^{-t/2}\sum_{i=0}^{d-1}{d-1 \choose i}R^{d-1-i}\Big(\int_{r^2/8}^{\infty}y^{i\over 2}e^{-y}\dd y + r^ie^{-r^2/8}\Big).
\end{aligned}
\end{equation}
The last integral is the incomplete gamma function $\Gamma(i/2+1,r^2/8)$ and its asymptotics, as $r\to\infty$, is well-known. Using the recurrence relation
$$
\Gamma(a+n,z)=(a)_n\Gamma(a,z)+z^ae^{-z}\sum\limits_{k=0}^{n-1}{\Gamma(a+n)\over\Gamma(a+k+1)}z^k
$$
from \cite[Equation (8.8.9)]{NIST} and the estimate $z^{1-a}e^z\Gamma(a,z)\leq 1$ for $z>0$ and $0<a\leq 1$ from  \cite[Inequality (8.10.1)]{NIST} we conclude that for $r\ge\sqrt{2d}>2$ and for $i\in\{0,1,\ldots,d-1\}$,
$$
\int_{r^2/8}^{\infty}y^{i/2}e^{-y}\,\dd y\ll_d r^ie^{-r^2/8}.
$$
Combining this with \eqref{eq:10.05.20_1} we obtain, that
\begin{equation}\label{eq_07}
P_1(\zeta_\beta)\ll_d(R+r)^{d-1}e^{-t/2-r^2/8}.
\end{equation}
In a next step, we observe that by estimate 2(a) of Lemma \ref{lm:BetaBoundaryBounds} for
$$
t=2\ln\ln\Big({1\over 1-\varepsilon/2}\Big)-2(d-1)\ln (R+r)-d\ln(2d+4)+\ln\pi
$$
and for $R\ge 1$ we have that
$$
P_2(\zeta_{\beta})\leq 1-\exp\big(-2^d(d/2+1)^{d/2}\pi^{-1/2}(R+r)^{d-1}e^{t/2}\big)\leq {\varepsilon\over 2}.
$$
It is also easy to ensure, that for sufficiently big $r$ (depending on $R$, $\varepsilon$ and $d$) we have that $r^2>-4t+16(d-1)a^2$. Thus, substituting expression above into \eqref{eq_07} we obtain for some constant $c_1(d)>0$, that
\begin{align*}
P_1(\zeta_{\beta})\leq {c_1(d)\over \ln\big({1\over 1-\varepsilon/2}\big)} (R+r)^{3(d-1)}e^{-r^2/8}.
\end{align*}
Since $\lim_{r\to\infty}(R+r)^{3(d-1)}e^{-r^2/8}=0$ and function on the right hand side of the last inequality is strictly decreasing for sufficiently large $r$, we may now choose $r_0(R,\varepsilon)>0$, depending on $R$, $\varepsilon$ and $d$ only in such a way that $P_1(\zeta_{\beta})\leq\varepsilon/2$ for any $r>r_0(R,\varepsilon)$.

By analogy, consider the case $\xi=\zeta^{\prime}_\beta$. Using the estimate 1(b) of Lemma \ref{lm:BetaBoundaryBounds} and choosing $a>0$ such that $(\sqrt{d-1}a)^{d-1}=(2(d+1))^{d/2}\sqrt{\pi}\Gamma({d+1\over 2})$
and taking into account that $R\ge 1$ we obtain
\begin{align*}
P_1(\zeta^{\prime}_\beta)\ll_d\sum_{i=0}^{d-1}{d-1 \choose i}R^{d-1-i}\sum_{y=0}^{\infty} (r^2/4+y+1)^{i\over 2}\exp(-e^{t/2+r^2/8+y/2-2(d-1)a^2}).
\end{align*}
Moreover, assuming that $r^2>-4-4t+16(d-1)a^2$ and using the fact that $1+x\leq e^x$ for $x>-1$ we have
\begin{align*}
P_1(\zeta^{\prime}_\beta)\ll_d\sum_{i=0}^{d-1}{d-1 \choose i}R^{d-1-i}\sum_{y=0}^{\infty} (r^2/4+y+1)^{i\over 2}e^{-t/2-r^2/8-y/2}.
\end{align*}
This expression was estimated above (see \eqref{eq_07}) and for $r^2>\max(-4-4t+16(d-1)a^2,\sqrt{2d})$ we get
\begin{equation}\label{eq_07_prime}
P_1(\zeta^{\prime}_\beta)\ll_d(R+r)^{d-1}e^{-t/2-r^2/8}.
\end{equation}
In the next step we will apply the estimate 2(b) of Lemma \ref{lm:BetaBoundaryBounds} with
$$
t=(3d-1)-{(3d-1)2^{d\over d-1}\over \pi^{1\over 2(d-1)}3^{d+1\over 2(d-1)}}(R+r)\Big(\ln\Big({1\over 1-\epsilon/2}\Big)\Big)^{-1/(d-1)}.
$$
It should be noted at this point, that $t<0$ and $r^2>-4-4t+16(d-1)a^2$ for sufficiently large $r$ (depending on $R$, $\varepsilon$ and $d$). Thus, for $\beta_0=3/2d-1/2$, $R>1$ and $d\ge 2$ we have
$$
P_2(\zeta^{\prime}_{\beta})\leq 1-\exp\Big(-{2^d\over\sqrt{\pi}}\big({3d-1\over 2d-2}\big)^{(d+1)/2}(R+r)^{d-1}\Big(1-{t\over 3d-1}\Big)^{-d+1}\Big)\leq {\varepsilon\over 2}.
$$
Combining this with \eqref{eq_07_prime} we get
\begin{align}\label{eq:26-01-21}
P_1(\zeta_{\beta})\leq c_{2}(d) (R+r)^{d-1}e^{-{r^2\over 8}+{c_3(d,\varepsilon) r\over 2}+{c_3(d,\varepsilon) R\over 2}}
\end{align}
for some constants $c_{2}(d)>0$ and $c_3(d,\varepsilon)>0$. Since for any $c_3(d,\varepsilon)>0$ we have
$$
\lim\limits_{r\to\infty}(R+r)^{d-1}e^{-{r^2\over 8}+{c_3(d,\varepsilon) r\over 2}+{c_3(d,\varepsilon) R\over 2}}=0,
$$
the right hand side in \eqref{eq:26-01-21} is decreasing for sufficiently large $r$, so that there exists $r^{\prime}_0(R,\varepsilon)>0$ such that $P_1(\zeta^{\prime}_{\beta})\leq\varepsilon/2$ for any $r>r^{\prime}_0(R,\varepsilon)$.

Finally, in case $\xi=\zeta$ we apply the estimate 1(c) of Lemma \ref{lm:BetaBoundaryBounds} with $a>0$ such that $(\sqrt{d-1}a)^{d-1}=2^{d/2-1}\sqrt{\pi}\Gamma({d+1\over 2})$. Together with the estimates $R\ge 1$ and $1+x\leq e^x$ for $x>-1$, and proceeding as above (see \eqref{eq_07}) with assumption that $r^2>\max(-4-4t+16(d-1)a^2,\sqrt{2d})$ we obtain
\begin{align*}
P_1(\zeta)\ll_d\sum_{i=0}^{d-1}{d-1 \choose i}R^{d-1-i}\sum_{y=0}^{\infty} (r^2/4+y+1)^{i\over 2}e^{-t/2-r^2/8-y/2}\ll_d(R+r)^{d-1}e^{-t/2-r^2/8}.
\end{align*}

Then by estimate 2(c) of Lemma \ref{lm:BetaBoundaryBounds} for
$$
t=2\ln\ln\Big({1\over 1-\varepsilon/2}\Big)-2\ln (R+r)^{d-1}-2d\ln 2+\ln\pi
$$
we get
$$
P_2(\zeta)\leq 1-\exp\big(-{2^d\over\sqrt{\pi}}(R+r)^{d-1}e^{t/2}\big)\leq {\varepsilon\over 2}.
$$
For sufficiently large $r$ it holds that $r^2>-4-4t+16(d-1)a^2$, implying that
$$
P_1(\zeta)\ll_d {1\over \ln\big({1\over 1-\varepsilon/2}\big)}(R+r)^{2(d-1)}e^{-{r^2\over 8}}.
$$
The last expression s decreasing for sufficiently large $r$ and vanishing, as $r\to\infty$. Thus, there exists $\tilde r_0(\epsilon, R)>0$, such that for $r>\tilde r_0(\epsilon, R)$ we have $P_1(\zeta)\leq {\varepsilon\over 2}$. In combination with \eqref{eq_06} this completes the argument.
\end{proof}

\subsection{Proof of the Theorem \ref{tm:CobvergenceGaussian}}\label{sec:ProofFinalStep}

We will give the proof only for $\widetilde{\sk{D}}_\beta$. The proof for $\widetilde{\sk{D}}^{\prime}_\beta$ works in the same way. To this end, let $C\in\cC$. From the construction of the Gaussian-Delaunay tessellation $\cD$ it is clear that the capacity functional of $\sk{D}$ satisfies $T(C)=T(\inter C)$. Since $C$ is bounded there exists $R\ge 1$ such that $C\subset B_{R}$. To simplify our notation, we denote the capacity functional of $\widetilde{\sk{D}}_\beta$ by $T_\beta(\,\cdot\,)$ and that of the limiting random closed set $\sk{D}$ by $T(\,\cdot\,)$. By the law of total probability we have
\begin{align*}
T_{\beta}(C)&=\PP\big(\widetilde{\sk{D}}_{\beta}\cap C \neq \varnothing \, \big| \, \widetilde{\sk{D}}_{\beta}\cap B_R = \sk{D}\cap B_R\big)\PP(\widetilde{\sk{D}}_{\beta}\cap B_R = \sk{D}\cap B_R)\\
&\qquad\qquad+\PP\big(\widetilde{\sk{D}}_{\beta}\cap C \neq \varnothing \, \big| \, \widetilde{\sk{D}}_{\beta}\cap B_R \neq \sk{D}\cap B_R\big)\PP(\widetilde{\sk{D}}_{\beta}\cap B_R \neq \sk{D}\cap B_R)\\
&= T(C) -\PP\big(\sk{D}\cap C \neq \varnothing \, \big| \, \widetilde{\sk{D}}_{\beta}\cap B_R \neq \sk{D}\cap B_R\big)\PP(\widetilde{\sk{D}}_{\beta}\cap B_R \neq \sk{D}\cap B_R) \\
&\qquad\qquad+\PP\big(\widetilde{\sk{D}}_{\beta}\cap C \neq \varnothing \, \big| \, \widetilde{\sk{D}}_{\beta}\cap B_R \neq \sk{D}\cap B_R\big)\PP(\widetilde{\sk{D}}_{\beta}\cap B_R \neq \sk{D}\cap B_R).
\end{align*}
Thus
\begin{equation}\label{eq_10}
T(C)-\PP(\widetilde{\sk{D}}_{\beta}\cap B_R \neq \sk{D}\cap B_R)\leq T_{\beta}(C)\leq T(C)+\PP(\widetilde{\sk{D}}_{\beta}\cap B_R \neq \sk{D}\cap B_R).
\end{equation}
Further, according to Lemma \ref{lm:StabRadius} for any $\varepsilon \in (0,1)$ and  $r>\max(r_0(\varepsilon,R), \tilde r_0(\varepsilon,R))$ we have
$$
\PP(\widetilde{\sk{D}}_{\beta}\cap B_R \neq \sk{D}\cap B_R)\leq \PP(\partial \Psi(\zeta_{\beta}) \neq \partial \Psi(\zeta) \text{ in } B_{R+r}\times \RR) + 2\varepsilon,
$$
and from Lemma \ref{lm:BetaBoundaryBounds} 1(a) and 1(c) for any $T>T_0(\varepsilon, R+r)$ we obtain
$$
\PP(\widetilde{\sk{D}}_{\beta}\cap B_R \neq \widetilde{\sk{D}}\cap B_R)\leq \PP(\partial \Psi(\zeta_{\beta}) \neq \partial \Psi(\zeta) \text{ in } B_{R+r}\times (-\infty,T]) + 4\varepsilon.
$$
It is clear now that the boundary of the paraboloid growth process $\Psi(\zeta_\beta)$ restricted to $B_{R+r}\times (-\infty,T]$ depends only on the restriction of the Poisson point process $\zeta_\beta$ to the set
$$
K(R+r,T)=(B_{R+r}\times (-\infty,T])\cup \Big(\bigcup\limits_{w\in\partial B_{R+r}}\Pi_{-,(w,T)}^{\downarrow}\Big),
$$
since only paraboloids $\Pi_{+,x}$ with $x\in K(R+r,T)$ can intersect $B_{R+r}\times (-\infty,T]$, see also Figure \ref{fig:Proof2}. Coupling the Poisson point processes $\zeta$ and $\zeta_{\beta}$ on a common probability space we obtain
$$
\PP(\widetilde{\sk{D}}_{\beta}\cap B_R \neq \sk{D}\cap B_R)\leq \PP(\zeta_{\beta}\cap K(R+r,T) \neq \zeta\cap K(R+r,T)) + 4\varepsilon.
$$
However, the last probability is equal to total variation distance between the Poisson point processes $\zeta$ and $\zeta_{\beta}$ (see \cite[Equation (5.12)]{Th00}), which according to \cite[Theorem 3.2.2]{Rei93} is bounded from above by a constant multiple of the $L_1$-norm of the difference of their densities. Thus,
\begin{align*}
&\PP(\widetilde{\sk{D}}_{\beta}\cap B_R \neq \sk{D}\cap B_R)\\
&\qquad\qquad\leq {3\over 2}\int_{-\infty}^T\int_{\{w\in\RR^{d-1}:\|w\|\leq \sqrt{T-s}+R+r\}}\Big|{e^{s/2}\over (2\pi)^{d/2}}-{c_{d,\beta}\over (2\beta)^{d/2}}{\bf 1}(s\geq -2\beta)\,\big(1+{s\over 2\beta}\big)^{\beta}\Big|\,\dint s\dint w  + 4\varepsilon\\
&\qquad\qquad={3\kappa_{d-1}\over 2}\int_{-\infty}^T(\sqrt{T-s}+R+r)^{d-1}\\
&\qquad\qquad\qquad\qquad\times\Big|{e^{s/2}\over (2\pi)^{d/2}}-{c_{d,\beta}e^{s/2}\over (2\beta)^{d/2}}+{c_{d,\beta}e^{s/2}\over (2\beta)^{d/2}}-{c_{d,\beta}\over (2\beta)^{d/2}}{\bf 1}(s\geq -2\beta)\,\big(1+{s\over 2\beta}\big)^{\beta}\Big|\,\dint s  + 4\varepsilon.
\end{align*}
Due to \eqref{eq_08} and the fact that $1+x\leq e^x$ for all $x>-1$ we obtain
\begin{align*}
&\PP(\widetilde{\sk{D}}_{\beta}\cap B_R \neq \sk{D}\cap B_R)\\
&\qquad\qquad\leq {3\kappa_{d-1}\over 2}\Big({c_{d,\beta}\over (2\beta)^{d/2}}-{1\over (2\pi)^{d/2}}\Big)\int_{-\infty}^T(\sqrt{T-s}+R+r)^{d-1}e^{s/2}\dint s\\
&\qquad\qquad\qquad\qquad+{3\kappa_{d-1}c_{d,\beta}\over 2(2\beta)^{d/2}}\int_{-\infty}^T(\sqrt{T-s}+R+r)^{d-1}\Big(e^{s/2}-{\bf 1}(s\geq -2\beta)\,\big(1+{s\over 2\beta}\big)^{\beta}\Big)\,\dint s  + 4\varepsilon\\
&\qquad\qquad=:I_1(\beta)+I_2(\beta)+4\varepsilon.
\end{align*}
To deal with $I_1(\beta)$ we compute
\begin{equation}\label{eq_09}
\begin{aligned}
\int_{-\infty}^T(\sqrt{T-s}+R+r)^{d-1}e^{s/2}\dint s&=\sum\limits_{i=0}^{d-1}{d-1\choose i}(R+r)^i\int_{-\infty}^T(T-s)^{(d-1-i)/2}e^{s/2}\,\dint s\\
&=e^{T/2}\sum\limits_{i=0}^{d-1}{d-1\choose i}(R+r)^i\int_{0}^{\infty}y^{(d-1-i)/2}e^{-y/2}\,\dint y\\
&=e^{T/2}\sum\limits_{i=0}^{d-1}{d-1\choose i}(R+r)^i\Gamma\Big({d+1-i\over 2}\Big)=:C_1<\infty.
\end{aligned}
\end{equation}
Then, since $\lim\limits_{\beta\to\infty}c_{d,\beta}\beta^{-d/2}=\pi^{-d/2}$ we conclude that there exists $\beta_1:=\beta_1(R,\varepsilon)$ such that for any $\beta>\beta_1$ we have $I_1(\beta)<\varepsilon$. On the other hand, by \eqref{eq_8} we obtain
\begin{align*}
&{3\kappa_{d-1}c_{d,\beta}\over 2(2\beta)^{d/2}}(\sqrt{T-s}+R+r)^{d-1}{\bf 1}(-2\beta\leq s\leq T )\,\big(1+{s\over 2\beta}\big)^{\beta}\\
&\qquad\qquad\qquad\qquad\leq {3\kappa_{d-1}(d/2+1)^{d/2}\over 2(2\pi)^{d/2}}(\sqrt{T-s}+R+r)^{d-1}{\bf 1}(s\leq T )e^{s/2}=:f(s).
\end{align*}
By \eqref{eq_09} the function $f$ is integrable and, hence, by dominated convergence theorem there exists $\beta_2:=\beta_2(R,\varepsilon)$ such that for any $\beta>\beta_2$ we have $I_2(\beta)<\varepsilon$. This implies that
$$
\PP(\widetilde{\sk{D}}_{\beta}\cap B_R \neq \sk{D}\cap B_R) \leq \varepsilon+\varepsilon+4\varepsilon = 6\varepsilon
$$
for $\beta>\max\{\beta_1,\beta_2\}$ and hence
$$
\lim\limits_{\beta\to\infty}\PP(\widetilde{\sk{D}}_{\beta}\cap B_R \neq \sk{D}\cap B_R)=0.
$$
Together with \eqref{eq_10} this shows that, as $\beta\to\infty$, for any $C\in\cC$ the value of the capacity functional $T_\beta(C)$ of $\widetilde{\sk{D}}_\beta$ converges to $T(C)$, the value of the capacity functional of $\sk{D}$. According to \eqref{eq:ConvergenceRACS} this proves the theorem and, at the same time, the stronger version described at the end of Section \ref{subsec:ConvergenceTess}. \hfill $\Box$

\section{Geometry of weighted typical cells of the Gaussian-Delaunay tessellation}\label{sec:Cells}

\subsection{Definition and probabilistic representation}

In this section we investigate the distribution of typical cells in the Gaussian-Delaunay tessellation $\cD$. Typical cells of $\beta$- and $\beta^{\prime}$-Delaunay tessellations have been considered in the first part of the article \cite{Part1}. We have already introduced there the concept of weighted typical cells in a more general context, which we briefly recall for convenience. In this section we let $\zeta$ be the Poisson point processes in $\RR^{d-1}\times\RR$ with intensity
$$
(v,h)\rightarrow {\gamma\over (2\pi)^{d/2}}e^{h/2},\quad \gamma >0.
$$
Let $\zeta^{*}$ be defined as in \eqref{eq:ApexProcess} or, in other words, we can consider $\zeta^{*}$ as the set of apexes of paraboloid facets of the paraboloid hull process $\Phi(\zeta)$:
$$
\zeta^{*}=\{(v,h)\colon (v,h)=\apex\Pi(x_1,\ldots,x_d),\, x_i\in\zeta,1\leq i\leq d,\,\conv(v_1,\ldots,v_d)\in\cD(\zeta)\}.
$$
Consider the random marked point process
\[
\mu_{\zeta}:=\sum\limits_{(v,h)\in \zeta^{*}}\delta_{(v,M)},\qquad M:=C((v,h), \zeta^{*}))-v
\]
in $\RR^{d-1}$, whose marks are the associated and suitably centred Laguerre cells. For a given parameter $\nu\in\RR$ we now define a probability measure $\PP_{\zeta,\nu}$ on $\cC'$, the space of non-empty compact subsets of $\RR^{d-1}$, by
\begin{equation}\label{eq_2}
\PP_{\zeta,\nu}(\,\cdot\,) := {1\over \lambda_{\zeta,\nu}}\EE\sum_{(v,M)\in\mu_{\zeta}}{\bf 1}(M\in\,\cdot\,){\bf 1}_{[0,1]^{d-1}}(v)\Vol(M)^\nu,
\end{equation}
where  $\lambda_{\zeta,\nu}\in[0,\infty]$ is the normalizing constant given by
\begin{equation}\label{eq:def_gamma_const}
\lambda_{\zeta,\nu}:= \EE\sum_{(v,M)\in\mu_{\zeta}}{\bf 1}_{[0,1]^{d-1}}(v)\Vol(M)^\nu.
\end{equation}
Note that $\lambda_{\zeta,\nu}$ might be infinite for some values of $\nu$. The proof of Theorem \ref{thm:GaussianDelaunayTypicalCell} will show that $0<\lambda_{\zeta,\nu}<\infty$, provided that $\nu\ge-1$. For such $\nu$, a random simplex $Z_{\nu}$ with distribution $\PP_{\nu}:=\PP_{\zeta,\nu}$ is called the {\bf  $\Vol^{\nu}$-weighted} (or just {\bf $\nu$-weighted}) {\bf typical cell} of the Gaussian-Delaunay tessellation $\cD$. The following two special cases are of particular interest:
\begin{itemize}
\item[(i)]  $Z_{0}$ coincides with the classical typical cell of $\cD$;
\item[(ii)]  $Z_{1}$ coincides with the volume-weighted typical cell of $\cD$, which has the same distribution as the almost surely uniquely determined cell containing the origin, up to translation.
\end{itemize}

The next theorem yields an explicit description of the distribution of the $\nu$-weighted typical cell of the Gaussian-Delaunay tessellation.

\begin{theorem}\label{thm:GaussianDelaunayTypicalCell}
Fix $d\geq 2$, $\nu\ge-1$ and $\gamma > 0$. Then for any Borel set $A\subset\cC'$ we have that
\begin{align*}
\PP_{\nu}(A) &= \widehat{\alpha}_{d,\nu}\int_{\RR^{d-1}}\dint y_1\ldots\int_{\RR^{d-1}}\dint y_d\,{\bf 1}_A(\conv(y_1,\ldots,y_d))\Delta_{d-1}(y_1,\ldots,y_d)^{\nu+1}\prod_{i=1}^de^{-\|y_i\|^2/2}
\end{align*}
where $\widehat{\alpha}_{d,\nu}$ is given by
\begin{align*}
\widehat{\alpha}_{d,\nu} := {1\over 2^{{(d-1)(d+\nu+1)\over 2}}\pi^{d(d-1)\over 2}}{((d-1)!)^{\nu+1}\over d^{\nu+1\over 2}}\prod_{j=1}^{d-1}{\Gamma({j\over 2})\over\Gamma({j+\nu+1\over 2})}.
\end{align*}
\end{theorem}

\begin{remark}\label{rem:GaussianDelaunayTypicalCell}
The previous theorem yields the following convenient probabilistic description of the $\nu$-weighted typical cell $Z_{\nu}$ of the Gaussian-Delaunay tessellation $\cD$. Let $Y_1,\ldots,Y_d$ be $d$ random vectors in $\RR^{d-1}$ whose joint density is proportional to
$$
\Delta_{d-1}(y_1,\ldots,y_d)^{\nu+1}\prod_{i=1}^de^{-\|y_i\|^2/2},\qquad y_1\in\RR^{d-1},\ldots,y_d\in\RR^{d-1}.
$$
Then the random simplex $\conv(Y_1,\ldots,Y_d)$ has the same distribution as $Z_{\nu}$. In other words, the $\nu$-weighted typical cell $Z_{\nu}$ of $\cD$ is a weighted Gaussian random simplex, where the weight is given by the $(\nu+1)$st volume power. As already explained in the introduction, this is the reason behind our terminology for the random tessellation $\cD$.
\end{remark}

\begin{remark}
It should be mentioned that the distribution $\PP_{\nu}$ of the $\nu$-weighted typical cell $Z_{\nu}$ of Gaussian-Delaunay tessellation is independent of the intensity parameter $\gamma$ of the underlying Poisson process $\zeta$. Indeed, changing $\gamma$ amounts to shifting the Poisson point process $\zeta$ in the vertical direction, which does not change the tessellation.
\end{remark}

\begin{proof}[Proof of Theorem \ref{thm:GaussianDelaunayTypicalCell}]
We start by observing that, up to a multiplicative constant, $\PP_{\nu}(A)$ is the same as
\begin{align*}
U(A) &:= {1\over d!}\EE\sum_{(x_1,\ldots,x_d)\in(\zeta)_{\neq}^d}{\bf 1}_A(\conv(v_1,\ldots,v_d)-z(x_1,\ldots,x_d))\\
&\qquad\times{\bf 1}_{[0,1]^{d-1}}(z(x_1,\ldots,x_d))\,{\bf 1}({\rm int}\,\Pi^\downarrow(x_1,\ldots,x_d)\cap\zeta=\varnothing)\,\Delta_{d-1}(v_1,\ldots,v_d)^\nu,
\end{align*}
where $z(x_1,\ldots,x_d)$ is the same generalized centre function as for the $\beta$-Poisson-Delaunay case for which we refer to \cite[Section 4.1]{Part1}. More concretely, $z(x_1,\ldots,x_d)\in \RR^{d-1}$  is the space coordinate of the apex of $\Pi(x_1,\ldots,x_d)$ defined as almost surely unique translate of the standard downward paraboloid $\Pi_{-}$ containing points $x_1=(v_1,h_1)\in \RR^d,\ldots,x_d=(v_d,h_d)\in \RR^d$ with affinely independent space coordinates $v_1,\ldots,v_d\in \RR^{d-1}$.
We now apply the multivariate Mecke's formula for Poisson point processes \cite[Corollary 3.2.3]{SW}, which yields
\begin{align*}
U(A) &= {\gamma^d\over d!(2\pi)^{d^2/2}}\int_{\RR^{d-1}}\dint v_1\ldots\int_{\RR^{d-1}}\dint v_d\int_\RR\dint h_1\ldots\int_\RR\dint h_d\\
&\qquad\times{\bf 1}_A(\conv(v_1,\ldots,v_d)-z(x_1,\ldots,x_d)){\bf 1}_{[0,1]^{d-1}}(z(x_1,\ldots,x_d))\\
&\qquad\times\PP({\rm int}\,\Pi^\downarrow(x_1,\ldots,x_d)\cap\zeta=\varnothing)\Delta_{d-1}(v_1,\ldots,v_d)^\nu\,e^{h_1/2+\ldots+h_d/2}.
\end{align*}
Let $(w,r)\in\RR^{d-1}\times\RR$ denotes the apex of the paraboloid $\Pi^\downarrow(x_1,\ldots,x_d)$. Then $w=z(x_1,\ldots,x_d)$, and we can write $v_i=w+z_i$ for $i=1,\ldots,d$ with uniquely determined $z_1,\ldots,z_d\in\RR^{d-1}$. Moreover, we have the relation $h_i=r-\|z_i\|^2$ for each $i=1,\ldots,d$. This allows us to introduce the transformation $T:\RR^{d-1}\times\RR\times(\RR^{d-1})^d\to(\RR^{d-1}\times\RR)^d$ by putting
$$
T(w,r,z_1,\ldots,z_d):=(w+z_1,r-\|z_1\|^2,\ldots,w+z_d,r-\|z_d\|^2) = (v_1,h_1,\ldots,v_d,h_d).
$$
It is not hard to check that for the absolute value $J(T)$ of the Jacobian of $T$ we have $J(T)=(d-1)!2^{d-1}\Delta_{d-1}(z_1,\ldots,z_d)$. In fact, using the convention that $|M|$ denotes the absolute value of the determinant of a matrix $M$, we have that
$$
J(T)=\left|
\begin{matrix}
E_{d-1} & 0 & E_{d-1} & 0 & \dots & 0\\
0 & 1 & -2z_1^{\top} & 0 &\dots & 0\\
E_{d-1} & 0 & 0 & E_{d-1} & \dots & 0\\
0 & 1 & 0 & -2z_2^{\top} & \dots & 0\\
\vdots & \vdots &  \vdots & \vdots &\ddots & \vdots \\
E_{d-1} & 0 & 0 & 0 & \dots & E_{d-1}\\
0 & 1 & 0 & 0 & \dots & -2z_d^{\top}\\
\end{matrix}
\right|,
$$
where $E_{d-1}$ is the $(d-1)\times(d-1)$ unit matrix and $z_1,\ldots,z_d$ are considered as column vectors. We can compute $J(T)$ by elementary column transformations as follows:
\begin{align*}
J(T) &= \left|
\begin{matrix}
0 & 0 & E_{d-1} & 0 & \dots & 0\\
2z_1^{\top} & 1 & -2z_1^{\top} & 0 &\dots & 0\\
0 & 0 & 0 & E_{d-1} & \dots & 0\\
2z_2^{\top} & 1 & 0 & -2z_2^{\top} & \dots & 0\\
\vdots & \vdots & \vdots & \vdots & \ddots & \vdots \\
0 & 0 & 0 & 0 & \dots & E_{d-1}\\
2z_d^{\top} & 1 & 0 & 0 & \dots & -2z_d^{\top}\\
\end{matrix}\right|
=\left|
\begin{matrix}
0 & E_{d(d-1)}\\
\begin{matrix}
2z_1^{\top} & 1 \\
\vdots & \vdots\\
2z_d^{\top} & 1\\
\end{matrix} & \mbox{\normalfont\Large\bfseries 0} \\
\end{matrix}\right|
=\left|
\begin{matrix}
1 & \ldots & 1\\
2z_1 &\ldots & 2z_d \\
\end{matrix}\right|,
\end{align*}
and the last expression is equal to $2^{d-1}(d-1)!\Delta_{d-1}(z_1,\ldots,z_d)$.
Applying the change of variables given by $T$ we obtain
\begin{align*}
U(A) &=
{\gamma^d\over d2^{d^2/2-d+1}\pi^{d^2/2}}\int_{\RR^{d-1}}\dint z_1\ldots\int_{\RR^{d-1}}\dint z_d\int_{\RR^{d-1}}\dint w\int_\RR\dint r\,{\bf 1}_A(\conv(z_1,\ldots,z_d)){\bf 1}_{[0,1]^{d-1}}(w)\\
&\quad\times\PP(\{(v,h)\in\RR^{d-1}\times\RR:h<-\|v-w\|^2+r\}\cap\zeta=\varnothing)\Delta_{d-1}(z_1,\ldots,z_d)^{\nu+1}\prod_{i=1}^de^{r/2-\|z_i\|^2/2}.
\end{align*}
Using the fact that $\zeta$ is a Poisson point process and inserting the precise form of its intensity measure, the probability in the last expression can be computed as follows:
\begin{align*}
&\PP(\{(v,h)\in\RR^{d-1}\times\RR:h<-\|v-w\|^2+r\}\cap\zeta=\varnothing)\\
&=\PP(\{(v,h)\in\RR^{d-1}\times\RR:h<-\|v\|^2+r\}\cap\zeta=\varnothing)\\
&=\exp\bigg(-{\gamma \over (2\pi)^{d/2}}\int_{\RR^{d-1}}\int_{\RR}{\bf 1}(h<-\|v\|^2+r)\,e^{h/2}\,\dint h\dint v\bigg)\\
&=\exp\bigg(-{\gamma\over (2\pi)^{d/2}}\int_{\RR^{d-1}}\int_{-\infty}^{-\|v\|^2+r}e^{h/2}\,\dint h\dint v\bigg)\\
&=\exp\bigg(-{2\gamma \over (2\pi)^{d/2}}e^{r/2}\int_{\RR^{d-1}}e^{-\|v\|^2/2}\,\dint v\bigg).
\end{align*}
The remaining integral is determined by introducing spherical coordinates in $\RR^{d-1}$, which yields
\begin{equation}\label{eq:15.01.20}
\int_{\RR^{d-1}}e^{-\|v\|^2/2}\,\dint v = (d-1)\kappa_{d-1}\int_0^\infty e^{-u^2/2}\,u^{d-2}\,\dint u = (2\pi)^{d-1\over 2}.
\end{equation}
Putting everything together and simplifying the constant leads to
\begin{align*}
&\PP(\{(v,h)\in\RR^{d-1}\times\RR:h<-\|v-w\|^2+r\}\cap\zeta=\varnothing) = \exp\bigg(-{\sqrt{2}\gamma\over\sqrt{\pi}}\,e^{r/2}\bigg).
\end{align*}
Thus, using the abbreviation $\widehat{m}_{d}:={\sqrt{2}\gamma\over\sqrt{\pi}}$, we find that
\begin{align*}
U(A) &= {\gamma^d I_d\over d2^{d^2/2-d+1}\pi^{d^2/2}}\int_{\RR^{d-1}}\dint z_1\ldots\int_{\RR^{d-1}}\dint z_d\,{\bf 1}_A(\conv(z_1,\ldots,z_d))\,\Delta_{d-1}(z_1,\ldots,z_d)^{\nu+1}\prod_{i=1}^de^{-\|z_i\|^2/2}\,,
\end{align*}
where $I_d:=\int_\RR e^{dr/2-\widehat{m}_d\,e^{r/2}}\,\dint r$.

Next, we compute the normalizing constant $U(\cC')$:
\begin{align*}
U(\cC') &=  {\gamma^d I_d\over d2^{d^2/2-d+1}\pi^{d^2/2}}\int_{\RR^{d-1}}\dint z_1\ldots\int_{\RR^{d-1}}\dint z_d\,\Delta_{d-1}(z_1,\ldots,z_d)^{\nu+1}\prod_{i=1}^de^{-\|z_i\|^2/2}\,.
\end{align*}
Now, we have to observe that the integral is the $(\nu+1)$st moment of the volume of a $(d-1)$-dimensional Gaussian simplex, up to a normalizing constant.  The  value of the moment can be found in~\cite[Equation~(70)]{miles_isotropic} where $\nu$ is required to be integer (see also~\cite[Theorem 2.3 (a)]{GKT17}, an extension to non-integer $\nu$ can be done as in the proof of Proposition~2.8 in~\cite{KTT}) and is given by
\begin{equation}\label{eq:moment_volume_gauss_simplex}
\begin{split}
{1\over (2\pi)^{d(d-1)/2}}\int_{(\RR^{d-1})^d} \Delta_{d-1}(y_1,\ldots,y_d)^{\nu+1}\,& \prod_{i=1}^de^{-\|y_i\|^2/2}\,\dint(y_1,\ldots,y_d)\\
&={d^{\nu+1\over 2}\over ((d-1)!)^{\nu+1}}2^{(\nu+1)(d-1)\over 2}\prod_{j=1}^{d-1}\bigg[{\Gamma({j+\nu+1\over 2})\over\Gamma({j\over 2})}\bigg].
\end{split}
\end{equation}
Thus,
$$
U(\cC') = \gamma^d I_{d}\, {2^{(d-1)(\nu+1)\over 2}\over 2^{1-{d\over 2}}\pi^{d\over 2}}{d^{\nu-1\over 2}\over ((d-1)!)^{\nu+1}}\prod_{j=1}^{d-1}\bigg[{\Gamma({j+\nu+1\over 2})\over\Gamma({j\over 2})}\bigg].
$$
As a consequence, simplification of the constants shows that
\begin{align*}
\PP_{\nu}(A) = {U(A)\over U(\cC')} &= \widehat{\alpha}_{d,\nu}\int_{\RR^{d-1}}\dint z_1\ldots\int_{\RR^{d-1}}\dint z_d\,{\bf 1}_A(\conv(z_1,\ldots,z_d))\Delta_{d-1}(z_1,\ldots,z_d)^{\nu+1}\prod_{i=1}^de^{-\|z_i\|^2/2}
\end{align*}
with $\widehat{\alpha}_{d,\nu}$ as in the statement of the theorem.
\end{proof}

\begin{remark}
In~\cite[Theorem 2.3 (a)]{GKT17}, Equation~\eqref{eq:moment_volume_gauss_simplex} is stated for $\nu\geq -1$ only, but in fact it remains true in the range $\nu>-2$ by analytic continuation. This implies that Theorem~\ref{thm:GaussianDelaunayTypicalCell} remains true in the range $\nu>-2$, but we shall not need this. This observation can be applied to slightly extend the range of validity of the  results of Section~\ref{subsec:expected_volume}.
\end{remark}

\subsection{Expected angle sums, volume moments and cell intensities}\label{subsec:expected_volume}

The probabilistic representation of weighted typical cells in Theorem \ref{thm:GaussianDelaunayTypicalCell} allows us to determine various geometric quantities associated with such cells. We start with the \textbf{expected angle sums} of $Z_{\nu}$. For a simplex $T=\conv(X_1,\ldots,X_d)\subset\RR^{d-1}$ we denote by
$$
\sigma_k(T) := \sum_{1\leq i_1<\ldots<i_k\leq d\atop F:=\conv(X_{i_1},\ldots,X_{i_k})}\beta(F,T)
$$
the sum of internal angles of $T$ at all its $k$-vertex faces $F$, see \cite[Chapter 6.5]{SW}. Further, by $\Sigma_{d-1}$ we denote a $(d-1)$-dimensional regular simplex. We recall that in \cite[Proposition 6.3]{Part1} of part I of this paper  we have seen that, as $\beta\to\infty$, the expected angle sums of the $\nu$-weighted typical cell of a $\beta$-Delaunay tessellation approach those of a regular simplex. Given Theorem \ref{tm:CobvergenceGaussian} it can now be anticipated that the expected angle sums of $\nu$-weighted typical cell of the limiting Gaussian-Delaunay tessellation coincide with that of a regular simplex. Our next result, which is essentially a consequence of Theorem \ref{thm:GaussianDelaunayTypicalCell}, confirms that this is indeed the case.

\begin{corollary}\label{cor:AngleSumGaussianDelaunay}
Let $Z_{\nu}$ be the $\nu$-weighted typical cell of the Gaussian-Delaunay tessellation with integer $\nu\ge-1$. Then, for all $k\in\{1,\ldots,d\}$, we have $\EE\sigma_k(Z_{\nu})=\sigma_k(\Sigma_{d-1})$.
\end{corollary}
\begin{proof}
Remark \ref{rem:GaussianDelaunayTypicalCell} yields that
$$
\EE\sigma_k(Z_{\nu}) = \EE\sigma_k(\conv(Y_1,\ldots,Y_d)),
$$
where $Y_1,\ldots,Y_d$ are random vectors in $\RR^{d-1}$ with joint density proportional to
$$
\Delta_{d-1}(y_1,\ldots,y_d)^{\nu+1}\prod_{i=1}^de^{-\|y_i\|^2/2},\qquad y_1,\ldots,y_d\in\RR^{d-1}.
$$
Taking the limit, as $\beta\to\infty$, in \cite[Remark 4.2]{beta_polytopes} (the argument is based on~\cite[Lemma~1.1]{beta_polytopes} and explained in the proof of Proposition~6.3 in~\cite{Part1}), we have that
$$
\EE\sigma_k(\conv(Y_1,\ldots,Y_d)) = \EE\sigma_k(\conv(N_1,\ldots,N_d)),
$$
where $N_1,\ldots,N_d$ are $d$ independent standard Gaussian random vectors in $\RR^{d-1}$.
Furthermore, in~\cite{GaussianSimplexAngles,GoetzeKabluchkoZap} it has been shown that the expected angle sum $\EE \sigma_k(N_1,\ldots,N_d)$ coincides with $\sigma_k(\Sigma_{d-1})$. Taking everything together, we arrive at
$$
\EE\sigma_k(Z_{\nu}) = \EE\sigma_k(\conv(Y_1,\ldots,Y_d)) = \EE\sigma_k(\conv(N_1,\ldots,N_d)) = \sigma_k(\Sigma_{d-1}),
$$
and the proof is complete.
\end{proof}

\begin{remark}
Note that \cite[Remark 4.2]{beta_polytopes}, which has been used in the previous proof, requires $\nu$ to be integer but it is natural to conjecture that it remains true without this assumption. This in turn would imply that Corollary \ref{cor:AngleSumGaussianDelaunay} also holds under the same circumstances.
\end{remark}

Next, we compute the moments of the volume of $Z_{\nu}$.

\begin{corollary}\label{cor:MomentsGaussianSimplex}
Let $\nu\ge -1$. Then, for all $s\ge-\nu-1$ we have that
$$
\EE\Vol(Z_{\nu})^s = 2^{{s(d-1)\over 2}}{d^{s/2}\over((d-1)!)^s}\prod_{j=1}^{d-1}{\Gamma({j+s+\nu+1\over 2})\over\Gamma({j+\nu+1\over 2})}.
$$
\end{corollary}
\begin{proof}
Using Theorem \ref{thm:GaussianDelaunayTypicalCell} together with \cite[Theorem 2.3 (a)]{GKT17} (restated in~\eqref{eq:moment_volume_gauss_simplex}) we have that
\begin{align*}
\EE\Vol(Z_{\nu})^s &= \widehat{\alpha}_{d,\nu}\int_{\RR^{d-1}}\dint y_1\ldots\int_{\RR^{d-1}}\dint y_d\,\Delta_{d-1}(y_1,\ldots,y_d)^{s+\nu+1}\,\prod_{i=1}^de^{-\|y_i\|^2/2}\\
&=\widehat{\alpha}_{d,\nu}\,(2\pi)^{d(d-1)\over 2}\,2^{(d-1)(s+\nu+1)\over 2}\,{d^{s+\nu+1\over 2}\over((d-1)!)^{s+\nu+1}}\prod_{j=1}^{d-1}{\Gamma({j+s+\nu+1\over 2})\over\Gamma({j\over 2})}.
\end{align*}
Plugging in the value for $\widehat{\alpha}_{d,\nu}$ completes the proof.
\end{proof}

From the previous corollary we can derive a probabilistic representation of the volume of $Z_\nu$, which is similar to that in \cite[Theorem 2.5(a)]{GKT17} for the volume of a Gaussian random simplex. To this end we recall that a random variable $X_k$ is said to have a $\chi^2$-distribution with $\kappa>0$ degrees of freedom, provided it has density
$$
t\mapsto {1\over 2^{\kappa\over 2}\Gamma({\kappa\over 2})}t^{{\kappa\over 2}-1}e^{-t/2},\qquad t>0,
$$
with respect to the Lebesgue measure on $(0,\infty)$.

\begin{corollary}
Let $\nu\ge-1$ and let $X_{j+\nu+1}$ be independent $\chi^2$-random variables with $j+\nu+1$ degrees of freedom, for $j\in\{1,\ldots,d-1\}$. Then the following random variables are identically distributed:
$$
((d-1)!)^2\Vol(Z_\nu)^2\qquad\text{and}\qquad  d \prod_{j=1}^{d-1}X_{j+\nu+1}.
$$
\end{corollary}
\begin{proof}
Let $Y$ be the random variable on the right hand side in the statement of the theorem. It is sufficient to show that, for $s>0$, $\EE[((d-1)!\Vol(Z_\nu))^{2s}]=\EE[Y^{s}]$ holds. From Corollary \ref{cor:MomentsGaussianSimplex} we have that
\begin{equation}\label{22-06-1}
\EE[((d-1)!\Vol(Z_\nu))^{2s}] = d^{s}\prod_{j=1}^{d-1}2^{s}{\Gamma({j+\nu+1\over 2}+s)\over\Gamma({j+\nu+1\over 2})}.
\end{equation}
On the other hand, for every $j\in\{1,\ldots,d-1\}$ one has that
$$
\EE[X_{j+\nu+1}^s] = 2^{s}{\Gamma({j+\nu+1\over 2}+s)\over\Gamma({j+\nu+1\over 2})}.
$$
Thus, from the assumed independence of $X_{\nu+2},\ldots,X_{\nu+d}$ it follows that
\begin{align}\label{22-06-2}
\EE[Y^{s}] = d^{s}\EE\Big[\prod_{j=1}^{d-1}X^s_{j+\nu+1}\Big] = d^{s}\prod_{j=1}^{d-1}\EE[X_{j+\nu+1}^s] = d^{s}\prod_{j=1}^{d-1}2^{s}{\Gamma({j+\nu+1\over 2}+s)\over\Gamma({j+\nu+1\over 2})}.
\end{align}
Comparison of \eqref{22-06-1} and \eqref{22-06-2} proves the claim.
\end{proof}

Corollary \ref{cor:MomentsGaussianSimplex} also allows us to compute the \textbf{intensity} $\gamma_j(\cD)$ \textbf{of $j$-dimensional faces} of $\cD$, that is, $\gamma_j(\cD)$ is the mean number of $j$-dimensional faces of $\cD$ per unit volume in $\RR^{d-1}$; see~\cite[p.~450 and \S~4.1]{SW} for an exact definition.

\begin{corollary}
For all $j\in\{0,1,\ldots,d-1\}$ we have that
$$
\gamma_j(\cD)={\sigma_{j+1}(\Sigma_{d-1})\over\EE\Vol(Z_{0})}\qquad\text{with}\qquad\EE\Vol(Z_0)=2^{d-1\over 2}{\sqrt{d}\over(d-1)!}\Gamma\Big({d+1\over 2}\Big).
$$
\end{corollary}
\begin{proof}
From \cite[Theorem 10.1.3]{SW} we have that $\gamma_j(\cD)=\gamma_{d-1}(\cD)\EE\sigma_{j+1}(Z_{0})$. The result thus follows from Corollary \ref{cor:AngleSumGaussianDelaunay} together with the fact that $\gamma_{d-1}(\cD)=1/\EE\Vol(Z_{0})$ from \cite[Equation (10.4)]{SW}. The value for $\EE\Vol(Z_{0})$ can be deduced from Corollary \ref{cor:MomentsGaussianSimplex} by taking $\nu=0$ and $s=1$.
\end{proof}

\subsection*{Acknowledgement}
ZK was supported by the DFG under Germany's Excellence Strategy  EXC 2044 -- 390685587, \textit{Mathematics M\"unster: Dynamics - Geometry - Structure}. CT and ZK were supported by the DFG via the priority program \textit{Random Geometric Systems}.

\end{document}